\documentclass[a4paper, 10pt]{amsart}
\usepackage{amsmath}
\usepackage{amssymb, color, hyperref}
\usepackage{tikz}
\usetikzlibrary{positioning}
\usepackage{amsfonts}
\usepackage{tikzscale}

\theoremstyle{plain}
\newtheorem{theorem}{\bf Theorem}[section]
\newtheorem{proposition}[theorem]{\bf Proposition}
\newtheorem{lemma}[theorem]{\bf Lemma}

\theoremstyle{definition}
\newtheorem{example}[theorem]{\bf Example}

\newtheorem{definition}[theorem]{\bf Definition}
\newtheorem{remark}[theorem]{\bf Remark}

\newcommand{\N}{\mathbb N}
\newcommand{\Z}{\mathbb Z}

\renewcommand{\t}{\, | \,}
\newcommand{\und}{\;\mbox{ and }\;}

 \DeclareMathOperator{\ord}{ord}

 \DeclareMathOperator{\supp}{supp}

\newcommand{\bdot}{\boldsymbol{\cdot}}
\newcommand{\bulletprod}[1]{\underset{#1}{\bullet}}
\newcommand{\red}{{\text{\rm red}}}

\numberwithin{equation}{section}

\begin{document}

\title[Product-one sequences]{On the algebraic and arithmetic structure \\ of the monoid of product-one sequences}

\address{Institute for Mathematics and Scientific Computing\\ University of Graz, NAWI Graz\\ Heinrichstra{\ss}e 36\\ 8010 Graz, Austria }
\email{junseok.oh@uni-graz.at}

\author{Jun Seok Oh}

\thanks{This work was supported by
the Austrian Science Fund FWF,  W1230 Doctoral Program Discrete Mathematics.}

\keywords{product-one sequences, Davenport constant, C-monoids, sets of lengths}

\subjclass[2010]{13A50, 20D60,  20M13}

\begin{abstract}
Let $G$ be a finite group. A finite unordered  sequence $S = g_1 \bdot \ldots \bdot g_{\ell}$ of terms from $G$, where repetition is allowed,   is a product-one sequence if its terms can be ordered such that their product equals $1_G$, the identity element of the group. As usual, we consider sequences as elements of the free abelian monoid $\mathcal F (G)$ with basis $G$, and we study the submonoid $\mathcal B (G) \subset \mathcal F (G)$ of all product-one sequences. This  is a finitely generated C-monoid, which  is a Krull monoid if and only if $G$ is abelian. In case of abelian groups,  $\mathcal B (G)$ is a well-studied object. In the present paper we focus on non-abelian groups, and  we study the class semigroup  and the arithmetic of $\mathcal B (G)$.
\end{abstract}

\maketitle

\bigskip
\section{Introduction} \label{1}
\bigskip

Let $G$ be a finite group. By a sequence over $G$, we mean a finite unordered sequence of terms from $G$, where the repetition of elements is allowed (the terminology stems from Arithmetic Combinatorics). A sequence $S$ is a product-one sequence if its terms can be ordered such that their product equals the identity element of the group. Clearly, juxtaposition of sequences is a  commutative operation on the set of sequences. As usual we consider sequences as elements of the free abelian monoid $\mathcal F (G)$ with basis $G$, and clearly the subset $\mathcal B (G) \subset \mathcal F (G)$ of all product-one sequences is a submonoid. The small Davenport constant $\mathsf d (G)$ is the maximal integer $\ell$ for which there is a sequence of length $\ell$ which has no product-one subsequence. The large Davenport constant $\mathsf D (G)$ is the maximal length of a minimal product-one sequence (by a minimal product-one sequence we mean an irreducible element in the monoid $\mathcal B (G)$).

Suppose that $G$ is abelian, and let us use additive notation in this case. Then product-one sequences are zero-sum sequences and their study is the main objective of Zero-Sum Theory (\cite{Ga-Ge06b, Gr13a}). The monoid $\mathcal B (G)$ is a Krull monoid with class group $G$ (apart from the exceptional case where $|G|=2$), and because it has intimate relationship with general Krull monoids having class group $G$, the study of $\mathcal B (G)$ is a central topic in factorization theory (\cite{Ge-HK06a, Ge09a}).

Although the abelian setting has been the dominant one, many of the combinatorial problems on sequences over abelian groups have also been studied in the nonabelian setting. For example, an  upper bound on the small Davenport constant was given   already in the 1970s \cite{Ol-Wh77}, and in recent years Gao's Theorem $\mathsf E (G) = |G|+\mathsf d (G)$ (\cite[Chapter 16]{Gr13a}) has been generalized to a variety of nonabelian groups (\cite{Ba07b, Ga-Lu08a, Ga-Li10b, Ha15a, Ha-Zh17b}). Fresh impetus came from invariant theory. If $G$ is abelian, then $\mathsf d (G)+1= \boldsymbol \beta (G) = \mathsf D (G)$, where $\boldsymbol \beta (G)$ is the Noether number.
If $G$ is nonabelian and has a cyclic subgroup of index two, then the Davenport constants and the Noether number have been explicitely determined (\cite{Cz-Do14a, Ge-Gr13a}), and it turned out that $\mathsf d (G)+1 \le \boldsymbol \beta (G) \le \mathsf D (G)$.
For a survey on the interplay with invariant theory we refer to \cite{Cz-Do-Ge16a}
and for recent progress to \cite{Gr13b, Cz-Do13c, Cz-Do15a, Cz-Do-Sz17}.

Let $F$ be a factorial monoid. A submonoid $H \subset F$ is a C-monoid if $H^{\times} = H \cap F^{\times}$ and the reduced class semigroup is finite. A commutative ring is called a C-ring if its monoid of regular elements is a C-monoid. To give an example of a C-ring, consider a Mori domain $R$. If its conductor $\mathfrak f = (R \colon \widehat R)$ is nonzero and the residue class ring $\widehat{R}/\mathfrak f$ is finite, then $R$ is a C-ring. For more on C-rings we refer to \cite{Re13a, Ge-Ra-Re15c}. The finiteness of the reduced class semigroup is used to derive arithmetical finiteness result for C-monoids (\cite{Ge-HK06a}). However, the structure of the (reduced) class semigroup has never been studied before.

The monoid $\mathcal B (G)$ of product-one sequences is a finitely generated C-monoid, which is Krull if and only if $G$ is abelian (in the Krull case, the class semigroup coincides with the class group which is isomorphic to $G$, apart from the exceptional case where $|G|=2$). After putting together the needed background in Section \ref{2}, we study the structure of the class semigroup of $\mathcal B (G)$ in Section \ref{3}. Among others we show that the unit group of the class semigroup is isomorphic to the center of $G$, and we reveal a subgroup of the class semigroup which is isomorphic to $G/G'$, which is the class group of the complete integral closure of $\mathcal B (G)$ (see Theorems \ref{3.8} and \ref{3.10}; $G'$ denotes the commutator subgroup of $G$). In Section \ref{4} we provide a complete description of the class semigroup for non-abelian groups of small order. In Section \ref{5} we study the arithmetic of $\mathcal B (G)$ and our focus is on sets of lengths. Among others we show that unions of sets of lengths  are finite intervals (Theorem \ref{5.5}).

\bigskip
\centerline{\it Throughout this paper, let $G$ be a multiplicatively written finite group with identity element $1_G \in G$.}

\bigskip
\section{Preliminaries} \label{2}
\bigskip

We denote by $\N$ the set of positive integers. For integers $a, b \in \Z$, $[a,b] = \{x \in \Z \mid a \le x \le b\}$ means the discrete interval. For every $n \in \N$, $C_n$ denotes a cyclic group of order $n$.
For an element $g \in G$, $\ord (g) \in \N$ is its order, and for a subset $G_0 \subset G$, $\langle G_0 \rangle \subset G$ denotes the subgroup generated by $G_0$.  We will use the following standard notation of group theory:
\begin{itemize}
\item $\mathsf Z (G) = \{g \in G \mid gx = xg \ \text{for all} \ x \in G\} \triangleleft G$ is the {\it center} of $G$,

\item $[x,y] = xyx^{-1}y^{-1} \in G$ is the {\it commutator} of the elements $x,y \in G$, and

\item $G' = [G,G] = \langle [x,y] \mid x,y \in G \rangle \triangleleft G$ is the    {\it commutator subgroup} of $G$.
\end{itemize}

\smallskip
\noindent
{\bf Semigroups.} All our semigroups are commutative and have an identity element.
Let $S$ be a semigroup.  We denote by $S^{\times}$ its group of invertible elements and by $\mathsf E (S)$ the set of all idempotents of $S$. $\mathsf E (S)$ is endowed with the Rees order $\leq$, defined by $e \leq f$ if $ef =e$. Clearly, $ef \leq e$ and $ef \le  f$ for all $e,f \in \mathsf E (S)$. If $E \subset \mathsf E (S)$ is a finite subsemigroup,  then $E$ has a smallest element.

By a {\it monoid}, we mean a  semigroup  which satisfies the cancellation laws. Let $H$ be a monoid. Then  $\mathsf q (H)$ denotes the quotient group of $H$ and $\mathcal A (H)$ the set of irreducibles (atoms) of $H$. The monoid $H$ is called {\it atomic} if every non-unit of $H$ can be written as a finite product of atoms. We say that $H$ is reduced if $H^{\times} = \{1\}$, and we denote by $H_{\red} =H/H^{\times} = \{aH^{\times} \mid a \in H \}$ the associated reduced monoid of $H$. A monoid $F$ is called {\it free abelian with basis $P \subset F$} if every $a \in F$ has a unique representation of the form
\[
a = \prod_{p \in P} p^{\mathsf v_p (a)} \quad \text{with} \quad \mathsf v_p (a)=0 \quad  \text{for almost all} \quad p \in P \,.
\]
If $F$ is free abelian with basis $P$, then $P$ is the set of primes of $F$, we set $F = \mathcal F (P)$, and denote by
\begin{itemize}
\item $|a| = \sum_{p \in P} \mathsf v_p (a)$ the {\it length} of $a$, and by
\item $\supp (a) = \{p \in P \mid \mathsf v_p (a)>0\}$ the {\it support} of $a$.
\end{itemize}
A monoid $F$ is factorial if and only if $F_{\red}$ is free abelian if and only if $F$ is atomic and every atom is a prime. We denote by
\begin{itemize}
\item $H' = \{ x \in \mathsf q (H) \mid \text{there is an $N \in \N$ such that } \ x^n \in H \ \text{for all} \ n \ge N\}$ the {\it seminormalization} of $H$,  by

\item $\widetilde H = \{ x \in \mathsf q (H) \mid x^N \in H \ \text{for some} \ N \in \N\}$ the {\it root closure} of $H$, and by

\item $\widehat H = \{ x \in \mathsf q (H) \mid \text{there is a $c \in H$ such that} \ cx^n \in H \ \text{for all} \ n \in \N \}$ the {\it complete integral closure} of $H$,
\end{itemize}
and observe that $H \subset H' \subset \widetilde H \subset \widehat H \subset \mathsf q (H)$. Then the monoid $H$ is called
\begin{itemize}
\item {\it seminormal} if $H = H'$ (equivalently,  if $x \in \mathsf q (H)$ and $x^2, x^3 \in H$, then
           $x \in H$),

\item {\it root closed} if $H = \widetilde H$,

\item {\it completely integrally closed} if $H =\widehat H$.
\end{itemize}
If $D$ is a monoid and $H \subset D$ a submonoid, then $H$ is a divisor-closed submonoid if $a \in H, b \in D$, and $b \mid a$ implies that $b \in H$.
A monoid homomorphism $\varphi \colon H \to D$ is said to be
\begin{itemize}
\item {\it cofinal} if for every $\alpha \in D$ there is an $a \in H$ such that $\alpha \mid \varphi (a)$.

\item a {\it divisor homomorphism} if $a, b \in H$ and $\varphi (a) \mid \varphi (b)$ implies that $a \mid b$.

\item a {\it divisor theory} if $D$ is free abelian, $\varphi$ is a divisor homomorphism, and for all $\alpha \in D$ there are $a_1, \ldots, a_m \in H$ such that $\alpha = \gcd ( \varphi (a_1), \ldots, \varphi (a_m) )$.

\item a {\it transfer homomorphism} if it satisfies the following two properties{\rm \,:}
      \begin{itemize}
      \item[{\bf (T\,1)\,}] $D = \varphi (H)D^{\times}$ and $\varphi^{-1} (D^{\times}) = H^{\times}$.
      \item[{\bf (T\,2)\,}] If $u \in H$, \ $b,\,c \in D$ \ and \ $\varphi (u) = bc$, then there exist \ $v,\,w \in H$ \ such that \ $u = vw$, \ $\varphi (v) D^{\times} = bD^{\times}$,  and  $\varphi (w)D^{\times}= cD^{\times}$.
      \end{itemize}
\end{itemize}

\medskip
\noindent
{\bf Transfer Krull monoids.} A monoid $H$ is said to be a {\it Krull monoid} if it satisfies one of the following equivalent conditions (see \cite[Theorem 2.4.8]{Ge-HK06a}):
\begin{enumerate}
\item[(a)] $H$ is completely integrally closed and satisfies the ACC on divisorial ideals.

\item[(b)] There is a divisor homomorphism $\varphi \colon H \to F$, where $F$ is free abelian monoid.

\item[(c)] $H$ has a divisor theory.
\end{enumerate}
A commutative domain $R$ is a Krull domain if and only if its multiplicative monoid of nonzero elements is Krull. Further examples of Krull monoids can be found in \cite{Ge-HK06a, Ge16c}. Let $H$ be a Krull monoid. Then a divisor theory $\varphi \colon H \to F$ is unique up to isomorphism and
\[
\mathcal C (H) = \mathsf q (F) / \mathsf q \big( \varphi (H)\big)
\]
is called the {\it class group} of $H$.

A monoid $H$ is said to be a {\it transfer Krull monoid} if it allows a transfer homomorphism to a Krull monoid (since in our setting all monoids are commutative,  \cite[Lemma 2.3.3]{Ba-Sm15} shows that the present definition coincides with the definition in  \cite{Ge16c}). The significance of transfer homomorphisms $\theta \colon H \to B$ stems from the fact that they allow to pull back arithmetical properties from the (simpler monoid) $B$ to the monoid $H$ (of original interest). In particular,  if $H$ is a transfer Krull monoid, then sets of lengths in $H$ coincide with sets of lengths in a Krull monoid (see Equation \eqref{transfer}).

Since the identity map is a transfer homomorphism, every Krull monoid is a transfer Krull monoid. For a list of transfer Krull monoids which are not Krull we refer to \cite{Ge16c}. To give one such example, consider a ring of integers $\mathcal O$ in an algebraic number field $K$, a central simple algebra $A$ over $K$, and a classical maximal $\mathcal O$-order $R$ of $A$. If every stably free left $R$-ideal is free, then the (non-commutative) semigroup of cancellative elements of $R$ is a transfer Krull monoid over a finite abelian group (\cite{Sm13a}).

\medskip
\noindent
{\bf Class semigroups and C-monoids.} (a detailed presentation can be found in \cite[Chapter 2]{Ge-HK06a}).
Let $F$ be a monoid and $H \subset F$ a submonoid. Two elements $y, y' \in F$
are called $H$-equivalent, denote $y \sim_{H} y'$, if $y^{-1}H \cap F = {y'}^{-1} H \cap
F$. $H$-equivalence is a congruence relation on $F$. For $y \in
F$, let $[y]_H^F$ denote the congruence class of $y$, and let
\[
\mathcal C (H,F) = \big\{ [y]_H^F \mid y \in F \big\} \quad \text{and} \quad \mathcal C^* (H,F) = \big\{ [y]_H^F \mid y \in (F \setminus F^{\times}) \cup \{1\} \big\} \,.
\]
Then $\mathcal C (H,F)$ is a commutative semigroup with unit element $[1]_H^F$
(called the {\it class semigroup} of $H$ in $F$) and $\mathcal C^* (H,F)
\subset \mathcal C (H,F)$ is a subsemigroup (called the {\it reduced
class semigroup} of $H$ in $F$).

As usual, class groups and class semigroups will both be written additively. The following lemma describes the relationship between class groups and class semigroups. Its proof is elementary and can be found in \cite[Propositions 2.8.7 and 2.8.8]{Ge-HK06a}

\begin{lemma} \label{2.1}~
Let $F$ be a monoid and $H \subset F$ be a submonoid.
\begin{enumerate}
\item If $H \subset F$ is cofinal, then the map $\theta \colon \mathcal C (H,F) \to \mathsf q (F)/\mathsf q (H)$, $[y]_H^F \mapsto y \mathsf q (H)$ for all $y \in F$, is an epimorphism. Moreover, $\theta$ is an isomorphism if and only if $H \hookrightarrow F$ is a divisor homomorphism.

\item If $\mathcal C (H,F)$ is a group, then $H \subset F$ is cofinal, and if $\mathcal C (H,F)$ is a torsion group, then  $H \hookrightarrow F$ is a divisor homomorphism.
\end{enumerate}
\end{lemma}

A monoid $H$ is called a {\it {\rm C}-monoid} if $H$ is a submonoid of a factorial monoid $F$ such that $H \cap F^{\times} = H^{\times}$ and $\mathcal C^* (H, F)$ is finite.
A commutative ring $R$ is a C-ring if its monoid of regular elements is a C-monoid. A Krull monoid is a C-monoid if and only if it has finite class group. We refer to \cite{Ge-HK06a, Re13a, Ge-Ra-Re15c, Ka16b} for more on C-monoids. To give an explicit example, consider a Mori ring $R$. If the conductor $\mathfrak f = (R \colon \widehat R)$ is nonzero and $\widehat{R}/\mathfrak f$ is finite, then  $R$ is a C-ring. We will need the
following lemma (for a proof see \cite[Theorem 2.9.11]{Ge-HK06a}).

\smallskip
\begin{lemma} \label{2.2}~
Let $F = F^{\times} \times \mathcal F (P)$ be factorial and $H \subset F$ be a {\rm C}-monoid. Suppose that $\mathsf v_p (H) \subset \N_0$ is a numerical monoid for all $p \in P$ (this is a minimality condition on $F$ which can always be assumed without restriction).
\begin{enumerate}
\item $\widehat H$ is a Krull monoid with nonempty conductor $(H \colon \widehat H)$ and finite class group $\mathcal C ( \widehat H)$.

\smallskip
\item The map \ $\partial \colon \widehat H \to \mathcal F(P)$, defined by
      \[
      \partial (a) = \prod_{p \in P} p^{\mathsf v_p(a)} \,,
      \]
      is a divisor theory, and there is an epimorphism \ $\mathcal C^* (H,F) \to \mathcal C(\widehat H)$.
\end{enumerate}
\end{lemma}


\bigskip
\section{Algebraic Properties of the monoid of product-one sequences} \label{3}
\bigskip

We introduce sequences over the finite group $G$. Our notation and terminology follows the articles \cite{Ge-Gr13a, Gr13b, Cz-Do-Ge16a}. Let $G_0 \subset G$ be a subset. The elements of the free abelian monoid $\mathcal F (G_0)$ will be called  {\it sequences} over $G_0$.  This terminology goes back to Arithmetic Combinatorics. Indeed, a sequence over $G_0$ can be viewed as a finite unordered sequence of terms from $G_0$, where the repetition of elements is allowed. In order to avoid confusion between the multiplication in $G$ and multiplication in $\mathcal F (G_0)$, we denote multiplication in $\mathcal F (G_0)$ by the boldsymbol $\bdot$ and we use brackets for all exponentiation in $\mathcal F (G_0)$. In particular, a sequence $S \in \mathcal F (G_0)$ has the form
\begin{equation} \label{basic}
S = g_1 \bdot \ldots \bdot g_{\ell} = \bulletprod{i\in [1,\ell]} g_i = \bulletprod{g \in G_0}g^{[\mathsf v_g (S)]} \in \mathcal F (G_0),
\end{equation}
where $g_1, \ldots, g_{\ell} \in G_0$ are the terms of $S$. Moreover,  if $S_1, S_2 \in \mathcal F (G_0)$ and $g_1, g_2 \in G_0$, then $S_1 \bdot S_2 \in \mathcal F (G_0)$ has length $|S_1|+|S_2|$, \ $S_1 \bdot g_1 \in \mathcal F (G_0)$ has length $|S_1|+1$, \ $g_1g_2 \in G$ is an element of $G$, but $g_1 \bdot g_2 \in \mathcal F (G_0)$ is a sequence of length $2$. If $g \in G_0$, $T \in \mathcal F (G_0)$, and $k \in \N_0$, then
\[
g^{[k]}=\underset{k}{\underbrace{g\bdot\ldots\bdot g}}\in \mathcal F (G_0) \quad \text{and} \quad T^{[k]}=\underset{k}{\underbrace{T\bdot\ldots\bdot T}}\in \mathcal F (G_0) \,.
\]
Let $S \in \mathcal F (G_0)$ be a sequence as in \eqref{basic}. Then we denote by
\[
\pi (S) = \{ g_{\tau (1)} \ldots  g_{\tau (\ell)} \in G \mid \tau\mbox{ a permutation of $[1,\ell]$} \} \subset G \quad \und \quad \Pi (S) = \underset{|T| \ge 1}{\bigcup_{T \t S}} {\pi}(T)  \subset G \,,
\]
the {\it set of products} and {\it subsequence products} of $S$, and it can easily be seen that $\pi (S)$ is contained in a $G'$-coset. Note that $|S|=0$ if and only if $S=1_{\mathcal F (G_0)}$, and in that case we use the convention that $\pi (S) = \{1_G\}$. The sequence $S$ is called
\begin{itemize}
\item a {\it product-one sequence} if $1_G \in \pi (S)$, and
\item {\it product-one free} if $1_G \notin \Pi (S)$.
\end{itemize}

\smallskip
\begin{definition} \label{3.1}~
Let $G_0 \subset G$ be a subset.
\begin{enumerate}
\item The submonoid
      \[
      \mathcal B (G_0) = \{S \in \mathcal F (G_0) \mid 1_G \in \pi (S) \} \subset \mathcal F (G_0)
      \]
      is called the {\it monoid of product-one sequences}, and $\mathcal A (G_0) = \mathcal A \big( \mathcal B (G_0) \big)$ denotes its set of atoms.

\item We call
      \[
      \mathsf D (G_0) = \sup \{ |S| \mid S \in \mathcal A (G_0) \} \in \N \cup \{\infty\}
      \]
      the {\it large Davenport constant} of $G_0$ and
      \[
      \mathsf d (G_0) = \sup \{ |S| \mid S \in \mathcal F  (G_0) \ \text{is product-one free} \} \in \N_0 \cup \{\infty\}
      \]
      the {\it small Davenport constant} of $G_0$.
\end{enumerate}
\end{definition}

Note that obviously the set $\mathcal B (G_0)$ is a multiplicatively closed subset of the commutative cancellative semigroup $\mathcal F (G_0)$ whence $\mathcal B (G_0)$ is indeed a monoid. Our object of interest is the monoid $\mathcal B (G)$ but for technical reasons we also need the submonoids $\mathcal B (G_0)$ for subsets $G_0 \subset G$.
The next lemma gathers some elementary properties. Its simple proof can be found in \cite[Lemma 3.1]{Cz-Do-Ge16a}.

\begin{lemma} \label{3.2}~
Let  $G_0 \subset G$ be a subset.
\begin{enumerate}
\item $\mathcal B (G_0) \subset \mathcal F (G_0)$ is a reduced finitely generated submonoid, $\mathcal A (G_0)$ is finite, and $\mathsf D (G_0) \le |G|$.

\smallskip
\item Let $S \in \mathcal F (G)$ be product-one free.
      \begin{enumerate}
      \smallskip
      \item If $g_0 \in \pi (S)$, then $g_0^{-1} \bdot S \in \mathcal A (G)$. In particular, $\mathsf d (G)+1 \le \mathsf D (G)$.

      \smallskip
      \item If $|S|=\mathsf d (G)$, then $\Pi (S) = G \setminus \{1_G\}$ and hence $\mathsf d (G) = \max \{ |S| \colon S \in \mathcal F (G) \ \text{with} \ \Pi (S) = G \setminus \{1_G\} \}$.
      \end{enumerate}

\smallskip
\item If $G$ is cyclic, then $\mathsf d (G)+1 = \mathsf D (G) = |G|$.
\end{enumerate}
\end{lemma}

\smallskip
The Davenport constant of finite abelian groups is a frequently studied invariant in Zero-sum Theory (for an overview see \cite{Ge-HK06a} and for recent progress  \cite{Sc11b, Sa-Ch14a}). For the Davenport constant of nonabelian groups we refer to \cite{Ge-Gr13a, Gr13b, Cz-Do-Sz17}.

\smallskip
\begin{lemma} \label{3.3}~
Let $G_0 \subset G$ be a subset. A submonoid $H \subset \mathcal B (G_0)$ is divisor-closed if and only if $H = \mathcal B (G_1)$ for a subset $G_1 \subset G_0$.
\end{lemma}

\begin{proof}
If $G_1 \subset G_0$, then clearly $\mathcal B (G_1) \subset \mathcal B (G_0)$ is a divisor-closed submonoid. Conversely, let $H \subset \mathcal B (G_0)$ be a divisor-closed submonoid. We set
\[
G_1 = \bigcup_{B \in H} \supp (B)
\]
and obtain that $H \subset \mathcal B (G_1)$. We have to show that equality holds. Let $S = g_1 \bdot \ldots \bdot g_{\ell} \in \mathcal B (G_1)$. Then for every $i \in [1, \ell]$ there is some $B_i \in H$ such that $g_i \in \supp (B_i)$. Then $B_1 \bdot \ldots \bdot B_{\ell} \in H$ and $S$ divides $B_1 \bdot \ldots \bdot B_{\ell}$, which implies that $S \in H$.
\end{proof}

\smallskip
\begin{proposition} \label{3.4}~
The following statements are equivalent{\rm \,:}
\begin{enumerate}
\item[(a)] $G$ is abelian.

\item[(b)] $\mathcal B (G)$ is a Krull monoid.

\item[(c)] $\mathcal B (G)$ is a transfer Krull monoid.
\end{enumerate}
\end{proposition}

\begin{proof}
The implications  (a) $\Rightarrow$ (b) and  (b) $\Rightarrow$ (c) are well-known and immediate. Indeed, if $G$ is abelian, then the embedding $\mathcal B (G) \hookrightarrow \mathcal F (G)$ is a divisor homomorphism whence $\mathcal B (G)$ is Krull. By definition, every Krull monoid is a transfer Krull monoid.

(c) $\Rightarrow$ (a) Assume to the contrary that $G$ is not abelian but $\mathcal B (G)$ is a transfer Krull monoid. Thus there is a transfer homomorphism $\theta_1 \colon \mathcal B (G) \to H$, where $H$ is a Krull monoid. Since there is a transfer homomorphism $\theta_2 \colon H \to \mathcal B (G_0)$, where $G_0 \subset \mathcal C (H)$ is a subset of the class group, and since the composition of transfer homomorphisms is a transfer homomorphism, we obtain a transfer homomorphism $\theta \colon \mathcal B (G) \to \mathcal B (G_0)$, where $G_0$ is a subset of an abelian group.

Since $G$ is non-abelian, there exist $g,h \in G$ such that $gh \neq hg$. Consider the sequence
\[
  U = g\bdot h\bdot g^{-1}\bdot (gh^{-1}g^{-1}) \in \mathcal A(G).
\]
From \cite[Proposition 3.2.3]{Ge-HK06a}, we have that $\theta(U) \in \mathcal A(G_{0})$, say $\theta(U) = a_{1}\bdot \ldots \bdot a_{\ell}$, where $a_{i} \in G_{0}$ for all $i \in [1,\ell]$.
Let $m = \textnormal{ord}\,(hgh^{-1}g^{-1}) \in \N$. Then
\[
  U^{[m]} = (g\bdot g^{-1})^{[m]}\bdot (h\bdot (gh^{-1}g^{-1}))^{[m]},
\]
and hence it follows that
\[
  \begin{aligned}
   (a_{1}\bdot \ldots \bdot a_{\ell})^{[m]} = &\,\, \theta(U^{[m]})\\
                                            = &\,\, \big(\theta(g\bdot g^{-1})\big)^{[m]}\bdot \theta\big( (h\bdot (gh^{-1}g^{-1}))^{[m]} \big).
   \end{aligned}
\]
Since $\mathcal B(G_{0}) \hookrightarrow \mathcal F(G_{0})$ is a divisor homomorphism (and so $\mathcal B (G_0)$ is root closed), the fact that
$\big(\theta(g\bdot g^{-1})\big)^{[m]}$ divides $ (a_{1}\bdot \ldots \bdot a_{\ell})^{[m]}$ in $\mathcal B (G_0)$ implies that $ \theta(g\bdot g^{-1})$ divides $ a_{1}\bdot \ldots \bdot a_{\ell}$ in $\mathcal B(G_{0})$.
Since $\theta(g\bdot g^{-1}) \und a_{1}\bdot \ldots \bdot a_{\ell} \in \mathcal A(G_{0})$, it follows that $\theta(g\bdot g^{-1}) = a_{1}\bdot \ldots \bdot a_{\ell} = \theta(U)$.
Thus $\theta\big( (h\bdot (gh^{-1}g^{-1}))^{[m]} \big) = 1_{\mathcal F(G_{0})}$, a contradiction.
\end{proof}

\smallskip
Clearly, if $|G|\le 2$, then $\mathcal B (G)$ is factorial whence it is both a Krull monoid (with trivial class group) and a C-monoid (with trivial class semigroup). In order to avoid trivial case distinctions, we exclude this case whenever convenient. By Proposition \ref{3.4}, $\mathcal B (G)$ is not Krull when $G$ is not abelian. The next proposition reveals that $\mathcal B (G)$ is a C-monoid in all cases and determines the class group of its complete integral closure.

\smallskip
\begin{proposition} \label{3.5}~
Suppose that $|G|\ge 3$.
\begin{enumerate}
\item If $G_0 \subset G$ is a subset, then $\mathcal B (G_0) \subset \mathcal F (G_0)$ is cofinal and $\mathcal B (G_0)$ is a {\rm C}-monoid.

\smallskip
\item   The embedding $\widehat{\mathcal B (G)} \hookrightarrow \mathcal F (G)$  is a divisor theory and the map
      \[
    \begin{aligned}
    \Phi \colon \mathcal C(\widehat{\mathcal B (G)}) = \mathsf q \big( \mathcal F (G) \big)/ \mathsf q \big( \mathcal B (G) \big) & \quad \longrightarrow \quad G / G' \\
         S \mathsf q (\mathcal B (G)) \quad  & \quad \longmapsto \quad g G' \,,
    \end{aligned}
    \]
    where $S \in \mathcal F (G)$ and $g \in \pi (S)$, is an isomorphism. Clearly, $\mathsf v_g \big( \mathcal B (G) \big) = \N_0$ for all $g \in G$.

\smallskip
\item   There is an epimorphism $\mathcal C \big( \mathcal B (G), \mathcal F (G) \big) \to G/G'$.
\end{enumerate}
\end{proposition}

\begin{proof}
1. and 2. Let $G_0 \subset G$. If $S = g_1 \bdot \ldots \bdot g_{\ell} \in \mathcal F (G_0)$ and $g \in \pi (S)$, then $S' = g^{-1} \bdot S \in \mathcal B (G_0)$, $S \mid S'$, and hence $\mathcal B (G_0) \subset \mathcal F (G_0)$ is cofinal. Let $k \in \N_0$. If $g \in G$ with $\ord (g) > 2$, then $U= g \bdot g^{-1} \in \mathcal B (G)$, $\mathsf v_g (U^{[k]}) = k$. If $g=1_G$, then $\mathsf v_g ( g^{[k]})=k$. If $\ord (g)=2$, then there is an $h \in G \setminus \{1_G, g\}$ and $U = g \bdot h \bdot (gh)^{-1} \in \mathcal B (G)$ and $\mathsf v_g (U^{[k]}) = k$. Thus in all cases we have $\mathsf v_g ( \mathcal B (G) ) = \N_0$.

It is easy to check that $\Phi$ is an isomorphism (details can be found in \cite[Theorem 3.2]{Cz-Do-Ge16a}). Since $\mathcal C(\widehat{\mathcal B (G)})$ is finite, $\mathcal C(\widehat{\mathcal B (G_0)})$ is finite and since $\mathcal B (G_0)$ is finitely generated, it is a C-monoid by \cite[Proposition 4.8]{Ge-Ha08b}.

\smallskip
3. This follows from 2. and from Lemma \ref{2.2}.2.
\end{proof}

\smallskip
We start with the study of the class semigroup and recall that, by definition, for two sequences $S, S' \in \mathcal F (G)$ the following statements are equivalent:
\begin{itemize}
\item $S \sim_{\mathcal B (G)} S'$.

\item For all $T \in \mathcal F (G)$, we have $S \bdot T \in \mathcal B (G)$ if and only if $S' \bdot T \in \mathcal B (G)$.

\item For all $T \in \mathcal F (G)$, we have $1_G \in \pi ( S \bdot T)$ if and only if $1_G \in \pi ( S' \bdot T)$.
\end{itemize}
If $S = g_1 \bdot \ldots \bdot g_{\ell} \in \mathcal B (G)$ such that $1_G = g_1 \ldots g_{\ell}$, then $1_G = g_i \ldots g_{\ell}g_1 \ldots g_{i-1}$ for every $i \in [1, \ell]$. We will use this simple fact without further mention. Moreover, $\sim$ means $\sim_{\mathcal B (G)}$ and we write $[S] = [S]_{\mathcal B (G)}^{\mathcal F (G)}$.

\smallskip
\begin{lemma} \label{3.6}~
Let $S \in \mathcal F (G)$.
\begin{enumerate}
\item If $S' \in \mathcal F (G)$ such that $S \sim S'$, then $\pi(S) = \pi(S')$.

\smallskip
\item Let $S' \in \mathcal F (G)$. In the following cases, $S \sim S'$ if and only if $\pi(S) = \pi(S')${\rm \,:}
      \begin{enumerate}
      \item[(a)] $S \und S'$ are sequences over the center of $G$.

      \smallskip
      \item[(b)] There is $g \in \pi(S)$ such that $\pi(S) = gG'$.
      \end{enumerate}

\smallskip
\item If $g, h \in G$ with $g \neq h$, then $g \nsim h$. In particular, $|G| \leq |\mathcal C \big(\mathcal B (G), \mathcal F (G)\big)|$.

\smallskip
\item If $g \in \mathsf{Z}(G) \und h \in G$, then $g\bdot h \sim gh$.

\smallskip
\item $|\pi (S)|=1$ if and only if $\langle \supp (S) \rangle$ is abelian.
\end{enumerate}
\end{lemma}

\begin{proof}
1.  Suppose that  $S' \in \mathcal F (G)$ with $S \sim S'$. Then for every $g \in G$ we obtain that
\[
g \in \pi(S) \quad \Leftrightarrow \quad g^{-1}\bdot S \in \mathcal B (G) \quad \Leftrightarrow \quad g^{-1}\bdot S' \in \mathcal B (G) \quad \Leftrightarrow \quad g \in \pi(S') \,,
\]
which implies that $\pi(S) = \pi(S')$.

\smallskip
2. From 1., it remains to verify the reverse implication. Suppose that $\pi(S) = \pi(S')$.

\smallskip
(a) Since $S, S' \in \mathcal F (\mathsf{Z}(G))$, we have $|\pi(S)| = |\pi(S')| = 1$, say $\pi(S) = \pi(S') = \{g\}$. Then for  every $T \in \mathcal F (G)$ we have
\[
  \pi(S\bdot T) = g\pi(T) \quad \und \quad \pi(S'\bdot T) = g\pi(T) \,,
\]
which implies that $S \sim S'$.

\smallskip
(b) Suppose that there is a $g \in \pi(S)$ such that $\pi(S) = gG'$. Let $T \in \mathcal F (G)$ such that $T\bdot S \in \mathcal B (G)$. Then we have
\[
  \pi(T)\pi(S) \subset \pi(T\bdot S) \subset G'.
\]
Since $\pi(S) = gG'$, there are $t \in \pi(T)$ and $e \in G'$ such that $tg = e$. Hence we obtain that
\[
  1_G = ee^{-1} = tge^{-1} \in \pi(T)\pi(S) = \pi(T)\pi(S') \subset \pi(T\bdot S').
\]
It follows that $T\bdot S' \in \mathcal B (G)$. By symmetry, we infer that $S \sim S'$.

\smallskip
3. Let $g, h \in G$. Then $g, h \in \mathcal F (G)$ and if  $g \sim h$, then 1. implies that  $\{g\} = \pi (g) = \pi (h) = \{h\}$.
Therefore, $g \neq h$ implies that  $g \nsim h$, and hence we obtain that $|G| \leq |\mathcal C \big(\mathcal B (G), \mathcal F (G)\big)|$.

\smallskip
4. This follows from the fact that $g \in \mathsf{Z}(G)$.

\smallskip
5. Obvious.
\end{proof}

\smallskip
\begin{lemma} \label{3.7}~
Let $S \in \mathcal F (G)$.
\begin{enumerate}
\item If $[S]$ is an idempotent of $\mathcal C \big(\mathcal B (G), \mathcal F (G)\big)$, then $\pi(S) \subset G'$ is a subgroup.

\smallskip
\item $[1_{\mathcal F (G)}] = \mathcal B \big(\mathsf{Z}(G)\big)$. In particular, if $g \in \mathsf{Z}(G)$, then $g^{[\ord (g)]} \sim 1_{\mathcal F (G)}$.

\smallskip
\item $[S]$ is the smallest idempotent in $\mathsf E \Big(\mathcal C \big(\mathcal B (G), \mathcal F (G)\big)\Big)$ (with respect to the Rees order) if and only if $\pi(S) = G'$.
\end{enumerate}
\end{lemma}

\begin{proof}
1. Suppose that $[S]$ is an idempotent of $\mathcal C \big(\mathcal B (G), \mathcal F (G)\big)$. Then $[S]=\big[S^{[2]}\big]$ whence $S \sim S^{[2]}$, and Lemma \ref{3.6}.1 implies that
\[
  \pi(S)\pi(S) \subset \pi\big(S^{[2]}\big) = \pi(S) \,.
\]
Thus $\pi(S) \subset G'$ is a subgroup.

\smallskip
2. Suppose that $S \in \mathcal B \big(\mathsf{Z}(G)\big)$. Then  $\pi(S) = \{1_{G}\}$, and for all $T \in \mathcal F (G)$ we have
\[
  \pi(S\bdot T) = \pi(T) = \pi ( 1_{\mathcal F (G)} \bdot T) \,,
\]
whence $S \sim 1_{\mathcal F (G)}$.

Conversely, suppose  that $S \sim 1_{\mathcal F (G)}$. Then $S \in \mathcal B (G)$, and we set
$S = g_{1}\bdot \ldots \bdot g_{\ell}$ such that $g_{1}\ldots g_{\ell} = 1_{G}$.
Assume to the contrary that there is an $i \in [1, \ell]$ such that $g_{i} \notin \mathsf{Z}(G)$, say $i=1$.
Then there is an element $h \in G$ such that $hg_{1} \neq g_{1}h$. Then $T = (hg_{1})\bdot (h^{-1}g_{1}^{-1}) \in \mathcal F (G) \setminus \mathcal B (G)$, but

\[
  1_{G} = g_{1}(hg_{1})(g_{2}\ldots g_{\ell})(h^{-1}g_{1}^{-1}) \in \pi(T \bdot S) \,.
\]
Since $S \sim 1_{\mathcal F (G)}$, we infer that $1_G \in \pi ( T \bdot 1_{\mathcal F (G)}) = \pi (T)$, a contradiction.

In particular, if $g \in \mathsf Z (G)$, then $g^{[\ord (g)]} \in \mathcal B \big(\mathsf{Z}(G)\big)$ and hence $g^{[\ord (g)]} \sim 1_{\mathcal F (G)}$.

\smallskip
3. First, we suppose that $\pi(S) = G'$. Then $[S]$ is an idempotent by Lemma \ref{3.6}.2,
and it remains to verify  that $[S]$ is the smallest idempotent of $\mathcal C \big(\mathcal B (G), \mathcal F (G)\big)$. Let $S' \in \mathcal F (G)$ such that $[S']$ is an idempotent. We have to show that $S \sim S \bdot S'$. Since $\pi (S') \subset G'$ is a subgroup by 1., $S, S' \in \mathcal B (G)$, and since $\pi (S \bdot S')$ is a $G'$-coset, it follows that
\[
G' \subset \pi (S)\pi (S') \subset \pi (S \bdot S') \subset G' \quad \text{whence} \quad G' = \pi (S \bdot S') \,.
\]
Again Lemma \ref{3.6}.2 implies that $S \sim S \bdot S'$.

To show the converse, suppose that $[S]$ is the smallest idempotent. We construct an element $S' \in \mathcal F (G)$ such that $\pi (S') = G'$. Then the proof above shows that $[S']$ is the smallest idempotent whence $S \sim S'$  and  $\pi (S)=\pi (S')=G'$.  We set $G' = \{ g_{1}, \ldots, g_{n}\}$, and for
each $i \in [1,n]$ we set
\[
g_{i} = \prod_{\nu=1}^{k_{i }} a_{i, \nu}b_{i,\nu}a^{-1}_{i, \nu}b^{-1}_{i, \nu } \,, \quad \text{ where $k_i \in \N$ and all } \quad  a_{i, \nu },b_{i, \nu} \in G \,,
\]
and define
\[
  S_{i} = \bulletprod{\nu \in [1,k_{i}]}(a_{i, \nu} \bdot b_{i, \nu} \bdot a^{-1}_{i, \nu} \bdot b^{-1}_{i, \nu}) \in \mathcal B (G) \,.
\]
Then obviously $\pi ( S_{1} \bdot \ldots \bdot S_{n})  = G'$.
\end{proof}

\smallskip
\begin{theorem} \label{3.8}~

\begin{enumerate}
\item Suppose that $G/G' = \{g_{0}G', \ldots, g_{k}G'\}$, and for each $i \in [0,k]$ let $S_i \in \mathcal F (G)$ such that $\pi(S_{i}) = g_{i}G'$.  Then the map
      \[
      \begin{aligned}
      G/G' & \quad \to \quad \big\{ [S_i] \mid i \in [0,k] \big\} = \mathcal C  \subset \mathcal C \big( \mathcal B (G), \mathcal F (G) \big) \\
      g_i G' & \quad \mapsto \quad [S_i]
      \end{aligned}
      \]
      is a group isomorphism. Thus $\mathcal C \big( \mathcal B (G), \mathcal F (G) \big)$ contains a subgroup isomorphic to the class group of $\widehat{\mathcal B (G)}$. Moreover, for any $i \in [0,k]$ and for any $S \in \mathcal F (G)$, we have $[S\bdot S_{i}] \in \mathcal C$.

\smallskip
\item The map
      \[
      \begin{aligned}
      \mathsf Z (G) & \quad \to \quad \mathcal C \big( \mathcal B (G), \mathcal F (G) \big)^{\times} \\
      g & \quad \mapsto \quad [g]
      \end{aligned}
      \]
      is a group isomorphism.
\end{enumerate}
\end{theorem}

\begin{proof}
1. We first verify the existence of such sequences $S_0, \ldots , S_k$. Since $\mathcal C \big(\mathcal B (G), \mathcal F (G)\big)$ is finite, the set $\mathsf E \Big(\mathcal C \big(\mathcal B (G), \mathcal F (G)\big)\Big)$ has the smallest element, say $[S]$. For each $i \in [0, k]$, we set $S_i = g_i \bdot S$. Since $\pi(S) = G'$ by Lemma \ref{3.7}.3, it follows that for each $i \in [0,k]$
\[
  g_iG' \subset \pi(g_i \bdot S) \subset g_iG' \,, \,\, \mbox{ whence } \, \pi(S_i) = \pi(g_i \bdot S) = g_iG' \,.
\]

We now use Lemma \ref{3.6}.2. If  $i, j \in [0,k]$, then $g_{i}g_{j} \in \pi(S_{i})\pi(S_{j}) \subset \pi(S_{i}\bdot S_{j})$, and hence $\pi(S_{i}\bdot S_{j}) = g_{i}g_{j}G' = g_{\ell}G' = \pi(S_{\ell})$ for some $\ell \in [0,k]$. Thus it follows that $\mathcal C$ is a subgroup of the class semigroup and the given map is an isomorphism. By Proposition \ref{3.5}.2, $G/G'$ is isomorphic to the class group of $\widehat{\mathcal B (G)}$.

Moreover, let $S \in \mathcal F (G)$, $g \in \pi (S)$, and $i \in [0,k]$. Then
\[
  gg_{i}G' = g\pi(S_{i}) \subset \pi(S)\pi(S_{i}) \subset \pi(S\bdot S_{i}) \,,
\]
whence $\pi(S\bdot S_{i}) = gg_{i}G' = g_{j}G' = \pi(S_{j})$ for some $j \in [0,k]$. Again by Lemma \ref{3.6}.2, we have $S\bdot S_{i} \sim S_{j}$, and thus $[S\bdot S_{i}] \in \mathcal C$.

\smallskip
2. Let $S \in \mathcal F (G)$ such that $[S] \in \mathcal C \big( \mathcal F (G), \mathcal B (G) \big)$ is invertible. Then there is an $S' \in \mathcal F (G)$ such that
\[
0_{\mathcal C ( \mathcal F (G), \mathcal B (G) )} = [1_{\mathcal F (G)}] = [S]+[S']= [S \bdot S'] \,,
\]
whence $S \bdot S' \sim 1_{\mathcal F (G)}$. Then Lemma \ref{3.7}.2 implies that $S \bdot S' \in \mathcal B \big( \mathsf Z (G) \big)$ whence $S,S' \in \mathcal F \big( \mathsf Z (G)\big)$. If $g \in \pi (S)$, then Lemma \ref{3.6}.2 implies that $S \sim g$. This proves that the given map is well-defined and surjective. Lemma \ref{3.6} (items 3 and 4) shows that the map is a monomorphism.
\end{proof}

\smallskip
Our next goal is a detailed investigation of the case where $|G'|=2$. We derive a couple of special properties which do not hold in general (Theorem \ref{3.10} and Remark \ref{3.11}).

\smallskip
\begin{lemma} \label{3.9}~
Suppose that $|G'| = 2$, and let $g \in G$ with $\ord(g)=n$.
\begin{enumerate}
\item We have $g \sim g^{[n+1]}$, and hence $\big[g^{[n]}\big] \in \mathcal C \big(\mathcal B (G), \mathcal F (G)\big)$ is idempotent.

\smallskip
\item If $k \in [1,n]$ is odd and $g^{k} \notin \mathsf{Z}(G)$, then $g^{[k]} \sim g^{k}$.

\smallskip
\item If $g, h \in G\setminus\mathsf{Z}(G)$ with $gh = hg$, then $g\bdot h \nsim gh$ provided that one of the following conditions holds{\rm \,:}
           \begin{enumerate}
           \smallskip
           \item[(a)] $gh \in \mathsf{Z}(G)$.

           \smallskip
           \item[(b)] There is $g_{0} \in G$ such that $g_{0}g \neq gg_{0} \und g_{0}h \neq hg_{0}$.
           \end{enumerate}
\end{enumerate}
\end{lemma}

\begin{proof}
1. Let  $T \in \mathcal F (G)$. Since $g^{n} = 1_{G}$,  $1_{G} \in \pi(g\bdot T)$ implies that $1_{G} \in \pi\big(g^{[n+1]}\bdot T\big)$.
Conversely, suppose that $1_{G} \in \pi\big(g^{[n+1]}\bdot T\big)$.
If every element in $\supp(T)$ commutes with $g$, then
\[
\pi\big(g^{[n+1]}\bdot T\big) = \pi(g\bdot T) \quad \text{and hence} \quad  1_{G} \in \pi(g \bdot T) \,.
\]
Now suppose  that there is at least one $h \in \supp(T)$ such that $hg \neq gh$. Then $\pi(g\bdot T)$ has at least two elements.
Since $|G'| = 2$ and
\[
  \pi(g\bdot T) \subset \pi\big(g^{[n+1]}\bdot T\big) \subset G',
\]
we obtain that $\pi(g\bdot T) = G'$ and hence $1_{G} \in \pi(g\bdot T)$. Thus $[g] = \big[g^{[n+1]}\big]$ and thus
\[
\big[g^{[n]}\big] + \big[g^{[n]}\big] = \big[g^{[2n]}\big] = \big[g^{[n+1]} \bdot g^{[n-1]}\big] = \big[g^{[n]}\big] \,.
\]

\smallskip
2. Let $k \in [1,n]$ be  odd, $g^k \notin \mathsf Z (G)$, and $T \in \mathcal F (G)$. If $1_G \in \pi \big(g^k \bdot T\big)$, then $1_G \in \pi \big(g^{[k]} \bdot T\big)$.   Conversely, suppose that $1_{G} \in \pi\big(T\bdot g^{[k]}\big)$. If $hg = gh$ for all $h \in \supp (T)$, then obviously $1_G \in \pi \big(g^k \bdot T\big)$.

Suppose there is an element $h \in \supp(T)$ such that $hg \neq gh$.
Then $\pi\big(T\bdot g^{[k]}\big) = G'$. We set $T = h_1 \bdot \ldots \bdot h_{\ell}$ with $h_1=h$ and consider the elements
\[
  g_{0} = h_{1}g^k h_{2}\ldots h_{\ell} \quad \und \quad g_{0}^{(1)} = g h_{1} g^{k-1} h_{2}\ldots h_{\ell} \,\, \mbox{ in } G'.
\]
Then $G' = \big\{g_0, g_0^{(1)}\big\}$. If $g_0=1$, then we are done.  Suppose  that $g_{0}^{(1)} = 1_{G}$.
Then
\[
  g_{0}^{(2)} = ggh_{1}g^{k-2}h_{2}\ldots h_{\ell} \neq g_{0}^{(1)}
\]
whence $g_{0}^{(2)} = g_{0}$, and  $hg^{2} = g^{2}h$. Thus we obtain
\[
  1_{G} = g_{0}^{(1)} = g_{0}^{(3)} = \ldots = g_{0}^{(k)} = g^{k}h_{1}\ldots h_{\ell} \in \pi\big(T\bdot g^{k}\big) \,.
\]

\smallskip
3. Let $g,h \in G\setminus\mathsf{Z}(G)$ with $gh = hg$.

(a) Suppose that $gh \in \mathsf Z (G)$ and assume to the contrary that $g\bdot h \sim gh$.
By Lemma \ref{3.7}.2, we infer that
\[
  (g\bdot h)^{[\ord(gh)]} \sim gh^{[\ord(gh)]} \sim 1_{\mathcal F (G)} \,.
\]
a contradiction to $[1_{\mathcal F (G)}] = \mathcal B \big(\mathsf{Z}(G)\big)$.

\smallskip
(b) Let $g_{0} \in G$ such that  $g_{0}g \neq gg_{0} \und g_{0}h \neq hg_{0}$. If
\[
  T = (g^{-1}g_{0})\bdot (h^{-1}g_{0}^{-1}) \in \mathcal F (G) \,,
\]
then $1_G \notin \pi \big(T\bdot (gh)\big)$ but $1_{G} = g(g^{-1}g_{0})h(h^{-1}g_{0}^{-1}) \in \pi(T\bdot g\bdot h)$.
This shows that $g\bdot h \nsim gh$.
\end{proof}

\smallskip
It is well-known that a finitely generated monoid is Krull if and only if it is root-closed. Thus if $G$ is non-abelian, then $\mathcal B (G)$ is not root-closed by Proposition \ref{3.4}. The next result shows that it is seminormal if $|G'|=2$.
An element $c$ of a semigroup is called a regular if $c$ lies in a subgroup of the semigroup, and the semigroup is called  \emph{Clifford} if every element is regular.

\smallskip
\begin{theorem} \label{3.10}~
Suppose that $|G'|=2$.
\begin{enumerate}
\item $\mathcal B (G)$ is seminormal.

\smallskip
\item $\mathcal C \big(\mathcal B (G), \mathcal F (G)\big)$ is a Clifford semigroup. In particular, if $S \in \mathcal F (G)$ with $\pi (S)=\{g\}$, then $[S]$ generates a cyclic subgroup of order $\ord (g)$.

\smallskip
\item $ |\mathcal C \big(\mathcal B (G), \mathcal F (G)\big)| \,\, \leq \,\, |\mathsf{Z}(G)| \,\, + \prod_{g \in G\setminus \mathsf{Z}(G)} \ord(g)$.
\end{enumerate}
\end{theorem}

\begin{proof}
1. Let $S \in \mathsf q \big(\mathcal B (G)\big)$ such that $S^{[2]},S^{[3]} \in \mathcal B(G)$. Since $S \in \mathsf q \big(\mathcal F (G)\big)$ with $S^{[2]}, S^{[3]} \in \mathcal F (G)$, we have that $S \in \mathcal F (G)$.
Let $g \in \pi(S)$. Since $S^{[2]},S^{[3]} \in \mathcal B (G)$, it follows  that
\[
  g^{2} \in \pi\big(S^{[2]}\big) \subset G' \quad \mbox{ and } \quad g^{3} \in \pi\big(S^{[3]}\big) \subset G' \,.
\]
Since $G'$ is a subgroup of $G$, we obtain $g \in G'$. If $g = 1_G$, then we are done.  Suppose that $g \ne 1_G$. Then $G' = \{1_G, g\}$.
If each two elements in $\supp(S)$ would commute, then
\[
  \pi(S) = \{ g \}, \quad \pi\big(S^{[2]}\big) = \{ g^{2} \}, \und \quad \pi\big(S^{[3]}\big) = \{ g^{3} \} \,,
\]
which implies $g^{2} = 1_G = g^{3} \und g = 1_G$, a contradiction.
Thus $\supp (S)$ contains two elements which do not commute, say $S = g_1 \bdot \ldots \bdot g_{\ell}, \,\, g_1g_2 \ne g_2g_1, \und g = g_1 \ldots g_{\ell}$.
Then $1_G = g_2g_1g_3 \ldots g_{\ell} \in \pi (S)$.

\smallskip
2. Let $S \in \mathcal F (G)$.
If $\pi(S)$ has two elements, then $[S]$ lies in the group given in Theorem \ref{3.8}.1.
Suppose that $\pi(S)$ has only one element, say $\pi(S) = \{g\}$ and  $\ord(g)=n$. We assert that $S \sim S^{[n+1]}$. Suppose  this holds true. Then  $\Big\{ [S], \ldots, \big[S^{[n]}\big] \Big\}$ is a cyclic subgroup of the class semigroup, and hence the assertion follows. Let $m \in [1,n]$ and $\big[S^{[m]}\big]$ the identity element of the subgroup. Then it is an idempotent of the class semigroup, and Lemma \ref{3.7}.1 shows that $\pi \big(S^{[m]}\big) = \{g^m\} \subset G'$ is a subgroup. This implies $m=n$.

Thus it remains to show that $S \sim S^{[n+1]}$. To do so,
let $T \in \mathcal F (G)$ be given. Since $S^{[n]} \in \mathcal B (G)$,
\[
  1_{G} \in \pi(T\bdot S) \quad \mbox{ implies that } \quad 1_{G} \in \pi\big(T\bdot S^{[n+1]}\big) \,.
\]
Conversely, suppose  that $1_{G} \in \pi\big(T\bdot S^{[n+1]}\big)$.
If every element of $\supp(T)$ commutes with every element of $\supp(S)$, then $\pi\big(T\bdot S^{[n+1]}\big) = \pi(T\bdot S)$ and thus $1_{G} \in \pi(T\bdot S)$.
If there are $t \in \supp(T) \und s \in \supp(S)$ such that $ts \neq s t $, then $|\pi(T \bdot S)| \geq 2$ and hence $1_{G} \in \pi(T \bdot S)$.

\smallskip
3. We set $G = \{g_{1}, \ldots, g_{n}\}$ and $\mathsf Z (G) = \{g_1, \ldots, g_k\}$ for some $k \in [1,n]$. Let $S \in \mathcal F (G)$. Then we can write $S$ of the form
\[
  S = g_{1}^{[k_{1}\ord(g_{1}) + r_{1}]}\bdot \ldots \bdot g_{n}^{[k_{n}\ord(g_{n}) + r_{n}]} \,,
\]
where $k_{1},\ldots,k_{n} \in \N_{0} \und r_{i} \in [1,\ord(g_{i})]$ for each $i \in [1,n]$.
By Lemma \ref{3.9}.1, we have
\[
  S \sim g_{1}^{[r_{1}]}\bdot \ldots \bdot g_{n}^{[r_{n}]} \,.
\]
By Lemma \ref{3.6}.4., $g_{1}^{[r_{1}]}\bdot \ldots \bdot g_{k}^{[r_{k}]} \sim g_0$, where $g_0 =  g_{1}^{r_{1}} \ldots  g_{k}^{r_{k}} \in \mathsf Z (G)$. Thus
\[
  S \sim g_0 \bdot g_{k+1}^{[r_{k+1}]}\bdot \ldots \bdot g_{n}^{[r_{n}]} \,,
\]
and hence the assertion follows.
\end{proof}

\medskip
\begin{remark} \label{3.11}~

1. Suppose that  $|G'|=2$.
Let $g, h \in G$ be distinct with  $\ord(g)=n$ and  $\ord(h)=m$. By Theorem \ref{3.10}.2, $[g] \subset \mathcal C \big( \mathcal B (G), \mathcal F (G)\big)$ generates a cyclic subgroup of order $n$ and $[h]$ generates a cyclic subgroup of order $m$.

Suppose that $\langle [g] \rangle \cap \langle [h] \rangle \neq \emptyset$. Then there are $i \in [1,n] \und j \in [1,m]$ such that $g^{[i]} \sim h^{[j]}$.
Let $m_{j} = \ord(h^{j}) = \frac{m}{\gcd(j,m)}$\,. Since $\big[h^{[m]}\big]$ is an idempotent, we have
\[
  g^{[im_{j}]} \sim h^{[m]} \,.
\]
By Lemma \ref{3.6}.1, $im_{j} = kn$ for some $k \in \N$. Since $\big[g^{[n]}\big]$ is also idempotent, we have
\[
  g^{[n]} \sim h^{[m]} \,.
\]
Again by Lemma \ref{3.6}.1, we obtain $gh = hg$ because $\mathcal B (G)$-equivalence is a congruence relation on $\mathcal F (G)$.

It follows that if $gh \neq hg$, then $\langle [g] \rangle \cap \langle [h] \rangle = \emptyset$, and for each $i \in [1,n] \und j \in [1,m]$,
\[
  \big[g^{[i]}\big] + \big[h^{[j]}\big] = \big[g^{[i]}\bdot h^{[j]}\big] \in \mathcal C \,,
\]
where $\mathcal C$ is the group given in Theorem \ref{3.8}.1.

\smallskip
2. Suppose that $\mathcal B (G)$ is seminormal and let $G_0 \subset G$ be a subset. Let $S \in \mathsf q \big(\mathcal B (G_0)\big)$ such that $S^{[2]},S^{[3]} \in \mathcal B(G_0)$. Since $\mathcal B (G)$ is seminormal, it follows that $S \in \mathcal B (G)$ and hence $S \in \mathcal B (G) \cap \mathsf q \big(\mathcal B (G_0)\big) = \mathcal B (G_0)$.
Thus $\mathcal B (G_0)$ is  seminormal.

\smallskip
3. Let $D_{2n} = \langle a,b \rangle = \{ 1_G, a, \ldots, a^{n-1}, b, ab, \ldots, a^{n-1}b \}$ be the dihedral group with $n \equiv 3 \mod 4$,
where $\ord(a) = n, \ord(b) = 2, \und a^{k}ba^{\ell}b = a^{k-\ell}$ for all $k,\ell \in \Z$. Then
\[
  S = a^{[\frac{n-1}{2}]}\bdot b^{[2]} \in \mathsf q \big(\mathcal B (D_{2n})\big) \setminus \mathcal B (D_{2n}), \,\, \mbox{ but } \, S^{[2]}, S^{[3]} \in \mathcal B (D_{2n}).
\]
Thus if $G$ contains $D_{2n}$ as a subgroup, then 2. shows that $\mathcal B (G)$ is not seminormal.
\end{remark}


\bigskip
\section{Examples of Class Semigroups for Non-abelian Groups of small order} \label{4}
\bigskip

In this section, we discuss the three smallest non-abelian groups and provide a complete description of the class semigroup. Among others we will see that the monoid of product-one sequences over the dihedral group with $6$ elements is not seminormal and the associated class semigroup is not Clifford.

\smallskip
\begin{example} \label{4.1}~
Let $G = Q_{8} = \{E, I, J, K, -E, -I, -J, -K \}$ be the quaternion group with the relations
\[
IJ = -JI = K, \quad JK = -KJ = I, \quad KI = - I K = J, \quad \und \quad IJK = -E \,.
\]
Then $\mathsf{Z}(G) = G' = \{E, -E\} \und G/G' \simeq C_{2} \oplus C_{2}$. Furthermore, $\mathsf d (G)=4$ and $\mathsf D (G)=6$ by \cite[Theorem 1.1]{Ge-Gr13a}.

Let $S \in \mathcal F (G)$. If $|\pi(S)| = 2$, then, by Theorem \ref{3.8}.1, $S$ is $\mathcal B (G)$-equivalent to an element in the group $\mathcal C$ which is isomorphic to $G/G'$.
We only consider the case $|\pi(S)| = 1$. By Lemma \ref{3.6}.5, $\langle \supp(S) \rangle$ is an abelian subgroup of $G$.
Suppose that $S \in \mathcal F \big(\langle I \rangle\big)$. Then $S$ has of the form
\[
  S = E^{[n_1]}\bdot I^{[n_2]}\bdot (-E)^{[n_3]}\bdot (-I)^{[n_4]} \,,
\]
where $n_1, \ldots, n_4 \in \N_0$. By Lemma \ref{3.6}.4 and Lemma \ref{3.9} (items 1 and 2), we have
\[
  S \sim I^{[n]} \,\, \mbox{ for some } \, n \in [1,4] \,.
\]
By symmetry, we obtain the same results in the case $S \in \mathcal F \big(\langle J \rangle\big) \und S \in \mathcal F \big(\langle K \rangle\big)$.
Moreover, If $S \in \mathcal F \big(\langle -E \rangle\big)$, then, by Theorem \ref{3.8}.2, $S$ is $\mathcal B (G)$-equivalent to an element in the group of units of the class semigroup $\mathcal C \big(\mathcal B (G), \mathcal F (G)\big)$ which is isomorphic to $\mathsf Z (G)$.
It follows that
\[
  |\mathcal C \big(\mathcal B (G), \mathcal F (G)\big)| = 18 \,.
\]

Figure 1 presents the subgroup lattice of the class semigroup.
Note that the subgroup lattice of $G$ is not preserved in the class semigroup $\mathcal C \big(\mathcal B (G), \mathcal F (G)\big)$. Furthermore, observe that the bottom elements in the lattice are all idempotent elements of the class semigroup.

\smallskip
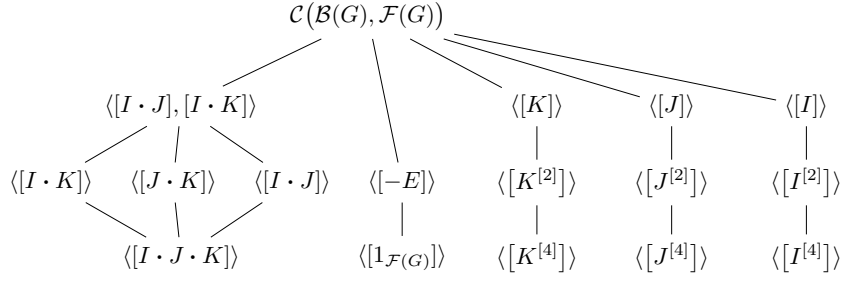
\begin{figure}[h]
\resizebox{.7\textwidth}{!}{%
\begin{tikzpicture}[node distance=2cm]
\title{Untergruppenverband der}
\node(C)                                  {$\mathcal C \big(\mathcal B (G), \mathcal F (G)\big)$};
\node(K)        [below right=1cm of C]    {$\langle [K] \rangle$};
\node(J)        [right of=K]              {$\langle [J] \rangle$};
\node(I)        [right of=J]              {$\langle [I] \rangle$};
\node(G/G')     [left=3.5cm of K]         {$\langle [I\bdot J], [I\bdot K] \rangle$};
\node(K2)       [below=.5cm of K]         {$\langle \big[K^{[2]}\big] \rangle$};
\node(J2)       [right of=K2]             {$\langle \big[J^{[2]}\big] \rangle$};
\node(I2)       [right of=J2]             {$\langle \big[I^{[2]}\big] \rangle$};
\node(-E)       [left of=K2]              {$\langle [-E] \rangle$};
\node(1)        [below=.5cm of -E]        {$\langle [1_{\mathcal F (G)}] \rangle$};
\node(K4)       [right of=1]              {$\langle \big[K^{[4]}\big] \rangle$};
\node(J4)       [right of=K4]             {$\langle \big[J^{[4]}\big] \rangle$};
\node(I4)       [right of=J4]             {$\langle \big[I^{[4]}\big] \rangle$};
\node(C1)       [left=.3cm of -E]         {$\langle [I\bdot J] \rangle$};
\node(C2)       [left=.3cm of C1]         {$\langle [J\bdot K] \rangle$};
\node(C3)       [left=.3cm of C2]         {$\langle [I\bdot K] \rangle$};
\node(C0)       [left=1.5cm of 1]         {$\langle [I\bdot J\bdot K] \rangle$};

\draw(C)       -- (I);
\draw(C)       -- (J);
\draw(C)       -- (K);
\draw(C)       -- (G/G');
\draw(C)       -- (-E);
\draw(I)       -- (I2);
\draw(J)       -- (J2);
\draw(K)       -- (K2);
\draw(-E)      -- (1);
\draw(K4)      -- (K2);
\draw(J4)      -- (J2);
\draw(I4)      -- (I2);
\draw(C1)      -- (G/G');
\draw(C2)      -- (G/G');
\draw(C3)      -- (G/G');
\draw(C0)      -- (C1);
\draw(C0)      -- (C2);
\draw(C0)      -- (C3);
\end{tikzpicture}
}
\caption{Subgroup Lattice of the Class Semigroup over $Q_{8}$}
\end{figure}
\end{example}

\smallskip
\begin{example} \label{4.2}~
Let $G = D_{8} = \langle a, b \t a^{4} = b^{2} = 1_{G} \und ba=a^{3}b \rangle = \{1_{G}, a, a^{2}, a^{3}, b, ab, a^{2}b, a^{3}b \}$ be the dihedral group of order $8$. Then $\mathsf{Z}(G) = G' = \langle a^{2} \rangle = \{1_{G}, a^{2}\} \und G/G' \simeq C_{2} \oplus C_{2}$. Furthermore, $\mathsf d (G)=4$ and $\mathsf D (G)=6$ by \cite[Theorem 1.1]{Ge-Gr13a}.

All arguments run along the same lines as the ones given in the previous example. However, in this case, there are two elements $g, h \in G\setminus \mathsf Z (G)$ such that $gh = hg$.
Consider the sequence $b\bdot a^{2}b \in \mathcal F (G)$ having $\pi(b\bdot a^{2}b) = \{ a^{2} \}$, and suppose that $b\bdot a^{2}b \sim S$ for some $S \in \mathcal F (G)$.
By Lemma \ref{3.6} (items 1 and 5), we have $\pi(S) = \{ a^{2} \}$ and hence $\langle \supp(S) \rangle \subset G$ is abelian subgroup containing $a^{2}$.
It follows that $\langle \supp(S) \rangle$ is one of $\langle a \rangle$, $\langle a^{2} \rangle$, $\langle a^{2}, b \rangle$ and $\langle a^{2}, ab \rangle$.

\smallskip
\noindent
CASE 1. $S \in \mathcal F \big(\langle a^{2} \rangle\big)$.

Then $S = a^{2}$ is only possible choice, but it never happen by Lemma \ref{3.9}.3.

\smallskip
\noindent
CASE 2. $S \in \mathcal F \big(\langle a \rangle\big)$.

Then $S = a^{[2]}$ is only possible choice by Lemma \ref{3.6}.4 and Lemma \ref{3.9} (items 1 and 2). Since $\mathcal B (G)$-equivalence is a congruence relation on $\mathcal F (G)$,
\[
  b\bdot a^{2}b \sim a^{[2]} \,\, \mbox{ implies that } \,\, b\bdot a^{2}b\bdot a \sim a^{[3]} \,.
\]
But $\{ a^{3} \} = \pi(a^{[3]}) = \pi(b\bdot a^{2}b\bdot a) = \{ a, a^{3} \}$, a contradiction.

In the case that $S \in \mathcal F \big(\langle a^{2}, ab \rangle\big)$, we can obtain a contradiction by the same argument of CASE 2.
It follows that $S = b\bdot a^{2}b$, and by Lemma \ref{3.6}.4,
\[
  S^{[2]} \sim b^{[2]} \,\, \mbox{ is an idempotent element of the class semigroup } \mathcal C \big(\mathcal B (G), \mathcal F (G)\big)
\]
with the relation $b^{[2]} \sim (a^{2}b)^{[2]} \sim S^{[2]}$.

For the sequence $T = ab\bdot a^{3}b \in \mathcal F (G)$, we can obtain the same result by above argument.
Thus $\big[T^{[2]}\big]$ is an idempotent element of the class semigroup $\mathcal C \big(\mathcal B (G), \mathcal F (G)\big)$ with the relation $(ab)^{[2]} \sim (a^{3}b)^{[2]} \sim T^{[2]}$.
It follows that
\[
  |\mathcal C \big(\mathcal B (G), \mathcal F (G)\big)| = 18 \,.
\]

Figure 2 presents the subgroup lattice of the class semigroup. Note that, as in the previous example, the subgroup lattice of $G$ is not preserved in the class semigroup $\mathcal C (\big(\mathcal B (G), \mathcal F (G)\big)$. Furthermore, observe that the bottom elements in the lattice are all idempotent elements of the class semigroup.

\smallskip
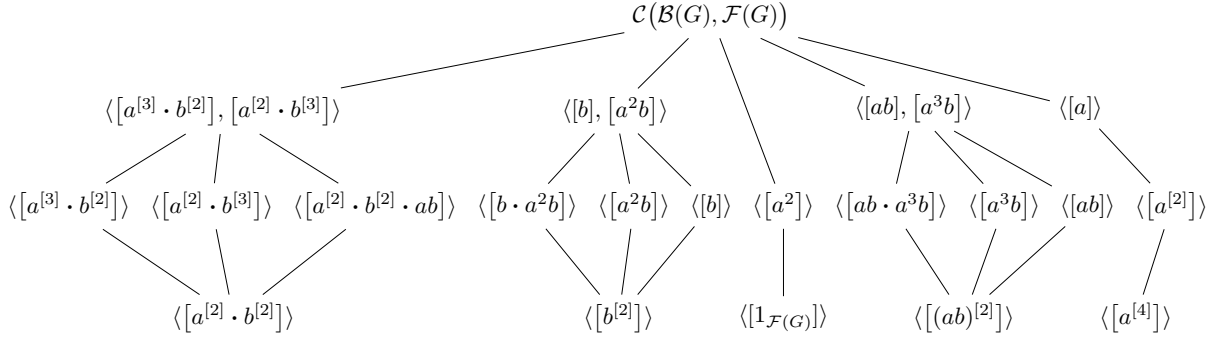
\begin{figure}[h]
\resizebox{1.00\textwidth}{!}{%
\begin{tikzpicture}[node distance=2cm]
\title{Untergruppenverband der}
\node(C)                                      {$\mathcal C \big(\mathcal B (G), \mathcal F (G)\big)$};
\node(ba2b)     [below left of= C]            {$\langle [b], \big[a^{2}b\big] \rangle$};
\node(aba3b)    [below right=1cm of C]        {$\langle [ab], \big[a^{3}b\big] \rangle$};
\node(G/G')     [left=3cm of ba2b]            {$\langle \big[a^{[3]}\bdot b^{[2]}\big], \big[a^{[2]}\bdot b^{[3]}\big] \rangle$};
\node(a)        [right=1cm of aba3b]          {$\langle [a] \rangle$};
\node(b)        [below right of=ba2b]         {$\langle [b] \rangle$};
\node(a2b)      [left=.05cm of b]             {$\langle \big[a^{2}b\big] \rangle$};
\node(bba2b)    [left=.05cm of a2b]           {$\langle \big[b\bdot a^{2}b\big] \rangle$};
\node(C1)       [left=.05cm of bba2b]         {$\langle \big[a^{[2]}\bdot b^{[2]}\bdot ab\big] \rangle$};
\node(C2)       [left=.05cm of C1]            {$\langle \big[a^{[2]}\bdot b^{[3]}\big] \rangle$};
\node(C3)       [left=.05cm of C2]            {$\langle \big[a^{[3]}\bdot b^{[2]}\big] \rangle$};
\node(a2)       [right=.05cm of b]            {$\langle \big[a^{2}\big] \rangle$};
\node(abba3b)   [right=.05cm of a2]           {$\langle \big[ab\bdot a^{3}b\big] \rangle$};
\node(a3b)      [right=.05cm of abba3b]       {$\langle \big[a^{3}b\big] \rangle$};
\node(ab)       [right=.05cm of a3b]          {$\langle [ab] \rangle$};
\node(a22)      [right=.05cm of ab]           {$\langle \big[a^{[2]}\big] \rangle$};
\node(1)        [below=1cm of a2]             {$\langle [1_{\mathcal F (G)}] \rangle$};
\node(b2)       [left=1cm of 1]               {$\langle \big[b^{[2]}\big] \rangle$};
\node(a2b2)     [right=1cm of 1]              {$\langle \big[(ab)^{[2]}\big] \rangle$};
\node(C0)       [left=4cm of b2]              {$\langle \big[a^{[2]}\bdot b^{[2]}\big] \rangle$};
\node(a4)       [right=1cm of a2b2]           {$\langle \big[a^{[4]}\big] \rangle$};

\draw(C)       -- (a);
\draw(C)       -- (ba2b);
\draw(C)       -- (G/G');
\draw(C)       -- (a2);
\draw(C)       -- (aba3b);
\draw(a)       -- (a22);
\draw(ba2b)    -- (b);
\draw(ba2b)    -- (a2b);
\draw(ba2b)    -- (bba2b);
\draw(aba3b)   -- (ab);
\draw(aba3b)   -- (a3b);
\draw(aba3b)   -- (abba3b);
\draw(a2)      -- (1);
\draw(b2)      -- (b);
\draw(b2)      -- (a2b);
\draw(b2)      -- (bba2b);
\draw(a2b2)    -- (ab);
\draw(a2b2)    -- (a3b);
\draw(a2b2)    -- (abba3b);
\draw(a4)      -- (a22);
\draw(C1)      -- (G/G');
\draw(C2)      -- (G/G');
\draw(C3)      -- (G/G');
\draw(C0)      -- (C1);
\draw(C0)      -- (C2);
\draw(C0)      -- (C3);
\end{tikzpicture}
}
\caption{Subgroup Lattice of the Class Semigroup over $D_{8}$}
\end{figure}
\end{example}

\smallskip
\begin{remark} \label{4.3}~
In general, the group of units $\mathcal C \big(\mathcal B (G), \mathcal F (G)\big)^{\times}$ of the class semigroup is not a direct factor. For example, let $G$ be the group described in Example \ref{4.1}. Assume to the contrary that $\mathcal C \big(\mathcal B (G), \mathcal F (G)\big)$ has a decomposition of the form
\[
 \mathcal C \big(\mathcal B (G), \mathcal F (G)\big) = \mathcal C_0 \times \mathcal C \big(\mathcal B (G), \mathcal F (G)\big)^{\times} \,,
\]
where $\mathcal C_0$ is a subsemigroup having $9$ elements. Figure 1. shows that there are three elements $x \in \mathcal C \big(\mathcal B (G), \mathcal F (G)\big)$ such that $x \neq x^{3}$. It follows that there exist three elements satisfying such property in the decomposition. Since $|\mathcal C \big(\mathcal B (G), \mathcal F (G)\big)^{\times}| = 2$, $\mathcal C_0$ has at least two such elements, whence $|\mathcal C_0| > 9$. Indeed, if $[I]$ and $[J]$ are in $\mathcal C_0$, then $[I\bdot J] = [I] + [J] \in \mathcal C_0$ and thus we obtain that $\mathcal C_0$ has at least $10$ elements. The similar argument works for the group described in Example \ref{4.2}.
\end{remark}

\smallskip
\begin{example} \label{4.4}~
Let $G = D_{6} = \langle a,b \t a^{3} = b^{2} = 1_{G} \und ba=a^{2}b \rangle = \{ 1_{G}, a, a^{2}, b, ab, a^{2}b \}$ be the dihedral group of order $6$. Then $\mathsf Z (G) = \{1_{G}\}$, $G' = \langle a \rangle = \{1_{G}, a, a^{2}\}$, and $G/G' \simeq C_2$. Furthermore, $\mathsf d (G)=3$ and $\mathsf D (G)=6$ by \cite[Theorem 1.1]{Ge-Gr13a}.

Let $S \in \mathcal F (G)$. If $|\pi(S)| = 3$, then, by Theorem \ref{3.8}.1, we obtain
\[
  S \sim a^{[2]}\bdot b^{[2]} \quad \mbox{ or } \quad S \sim a\bdot a^{2}\bdot b \,,
\]
where $\pi(a^{[2]}\bdot b^{[2]}) = \{ 1_G, a, a^{2} \} \,\, \und \,\, \pi(a\bdot a^{2}\bdot b) = \{ b, ab, a^{2}b \}$.

If $|\pi(S)| = 2$, then $S$ is $\mathcal B (G)$-equivalent with one of the following sequences :
\begin{equation} \label{length 2}
 \begin{aligned}
  a\bdot b^{[n]}, \quad a\bdot (ab)^{[n]}, \quad a\bdot (a^{2}b)^{[n]}, \quad a^{2}\bdot b^{[n]}, \quad a^{2}\bdot (ab)^{[n]}, \quad a^{2}\bdot (a^{2}b)^{[n]} \,,\\
  b^{[n]}\bdot ab, \quad b\bdot (ab)^{[n]}, \quad b^{[n]}\bdot a^{2}b, \quad b\bdot (a^{2}b)^{[n]}, \quad (ab)^{[n]}\bdot a^{2}b, \quad ab\bdot (a^{2}b)^{[n]} \,,
  \end{aligned}
\end{equation}
where $n \in \N$.

We now start with the following claims.

\smallskip
\noindent
\[
  \begin{aligned}
   {\bf CLAIM. A:} \hspace{.5cm} & b^{[2]} \sim b^{[4]}, \quad (ab)^{[2]} \sim (ab)^{[4]}, \quad (a^{2}b)^{[2]} \sim (a^{2}b)^{[4]}, \quad a^{[2]} \sim a^{[5]}, \quad (a^{2})^{[2]}  \sim (a^{2})^{[5]} \,,\\
                                 & (a^{2})^{[2]} \sim a^{[4]}, \quad (a^{2})^{[3]} \sim a^{[3]}, \quad (a^{2})^{[4]} \sim a^{[2]}, \quad a\bdot a^{2} \sim a^{[3]} \,,\\
                                 & a\bdot b \sim a^{2}\bdot b \sim a\bdot b^{[3]} \sim b^{[2]}\bdot ab \sim b^{[2]}\bdot a^{2}b \,,\\
                                 & a\bdot ab \sim a^{2}\bdot ab \sim a\bdot (ab)^{[3]} \sim ab^{[2]}\bdot a^{2}b \sim ab^{[2]}\bdot b \,, \\
                                 & a\bdot a^{2}b \sim a^{2}\bdot a^{2}b \sim a\bdot (a^{2}b)^{[3]} \sim (a^{2}b)^{[2]}\bdot ab \sim (a^{2}b)^{[2]}\bdot b \,.
  \end{aligned}
\]
\noindent
\[
  \begin{aligned}
  {\bf CLAIM. B:} \hspace{.5cm} & b \nsim b^{[3]}, \quad ab \nsim (ab)^{[3]}, \quad a^{2}b \nsim (a^{2}b)^{[3]}, \quad a \nsim a^{[4]}, \quad a^{[2]} \nsim a^{2} \,, \hspace{2.8cm}\\
                                & b\bdot ab \nsim b\bdot a^{2}b, \quad ab\bdot a^{2}b, \quad a\bdot b^{[2]}, \quad a\bdot (ab)^{[2]}, \quad a\bdot (a^{2}b)^{[2]} \,,\\
                                & b\bdot a^{2}b \nsim ab\bdot a^{2}b, \quad a\bdot b^{[2]}, \quad a\bdot (ab)^{[2]}, \quad a\bdot (a^{2}b)^{[2]} \,,\\
                                & ab\bdot a^{2}b \nsim a\bdot b^{[2]}, \quad a\bdot (ab)^{[2]}, \quad a\bdot (a^{2}b)^{[2]} \,.
  \end{aligned}
\]

\smallskip
Suppose the Claims hold true. Then we obtain that
\[
  |\mathcal C \big(\mathcal B (G), \mathcal F (G)\big)| = 26 \quad \und \quad \mathcal C \big(\mathcal B (G), \mathcal F (G)\big) = G_1 \cup G_2 \cup G_3 \cup G_4 \,,
\]
where
\begin{itemize}
\item  $G_1 = \Big\{ [b], \big[b^{[2]}\big], \big[b^{[3]}\big], [ab], \big[(ab)^{[2]}\big], \big[(ab)^{[3]}\big], [a^{2}b], \big[(a^{2}b)^{[2]}\big], \big[(a^{2}b)^{[3]}\big], [a], \big[a^{[2]}\big], \big[a^{[3]}\big], \big[a^{[4]}\big], \big[a^{2}\big] \Big\}$,

\item $G_2 = \Big\{ [a\bdot b], [a\bdot ab], \big[a\bdot a^{2}b\big], \big[a\bdot b^{[2]}\big], \big[a\bdot (ab)^{[2]}\big], \big[a\bdot (a^{2}b)^{[2]}\big], [b\bdot ab], \big[b\bdot a^{2}b\big], \big[ab\bdot a^{2}b\big] \Big\}$,

\item $G_3 = \Big\{ \big[a^{[2]}\bdot b^{[2]}\big], \big[a\bdot a^{2}\bdot b\big] \Big\}$ (that is isomprphic to $G/G'$), and

\item $G_4 = \Big\{ [1_{\mathcal F (G)}] \Big\}$ \big(that is isomorphic to $\mathsf Z (G)$\big).
\end{itemize}

Figure 3 presents the subgroup lattice of the class semigroup. Note that, as in the previous example, the subgroup lattice of $G$ is not preserved in the class semigroup $\mathcal C \big(\mathcal B (G), \mathcal F (G)\big)$. Furthermore, observe that the elements in the set
\[
  \Big\{ [b], [ab], \big[a^{2}b\big], [a], \big[a^{2}\big] \Big\} \cup G_2
\]
are not regular, whence $\mathcal C \big(\mathcal B (G), \mathcal F (G)\big)$ is not Clifford semigroup, and the bottom elements in the lattice are all idempotent elements of the class semigroup.

\smallskip
\begin{figure}[h]
\resizebox{.7\textwidth}{!}{%
\begin{tikzpicture}[node distance=2cm]
\title{Untergruppenverband der}
\node(C)                                               {$\mathcal C \big(\mathcal B (G), \mathcal F (G)\big)$};
\node(a2)       [below right=.05cm and 2cm of C]       {$\langle \big[a^{[2]}\big] \rangle$};
\node(a2b3)     [below right of=C]                     {$\langle \big[(a^{2}b)^{[3]}\big] \rangle$};
\node(ab3)      [below left of=C]                      {$\langle \big[(ab)^{[3]}\big] \rangle$};
\node(b3)       [left=1cm of ab3]                      {$\langle \big[b^{[3]}\big] \rangle$};
\node(G/G')     [left=1cm of b3]                       {$\langle \big[a\bdot a^{2}\bdot b\big] \rangle$};
\node(1)        [below=2.45cm of C]                    {$\langle [1_{\mathcal F (G)}] \rangle$};
\node(a2b2)     [right=.25cm of 1]                     {$\langle \big[(a^{2}b)^{[2]}\big] \rangle$};
\node(a3)       [right=.4cm of a2b2]                   {$\langle \big[a^{[3]}\big] \rangle$};
\node(ab2)      [left=.25cm of 1]                      {$\langle \big[(ab)^{[2]}\big] \rangle$};
\node(b2)       [left=.5cm of ab2]                     {$\langle \big[b^{[2]}\big] \rangle$};
\node(C0)       [left=1cm of b2]                       {$\langle \big[a^{[2]}\bdot b^{[2]}\big] \rangle$};

\draw(C)       -- (a2);
\draw(C)       -- (a2b3);
\draw(C)       -- (G/G');
\draw(C)       -- (ab3);
\draw(C)       -- (b3);
\draw(C)       -- (1);
\draw(a2b2)    -- (a2b3);
\draw(ab2)     -- (ab3);
\draw(b2)      -- (b3);
\draw(C0)      -- (G/G');
\draw(a3)      -- (a2);
\end{tikzpicture}
}
\caption{Subgroup Lattice of Class Semigroup over $D_{6}$}
\end{figure}
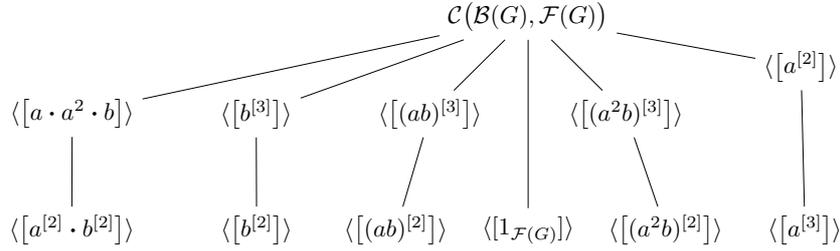

\smallskip
\begin{proof}[Proof of the Claim]
We only show the first assertions, and others can be proved by the same way. Let $T \in \mathcal F (G)$.
Since $b^{2} = 1_G$, it suffices to show that if $b^{[4]}\bdot T \in \mathcal B (G)$, then $b^{[2]}\bdot T \in \mathcal B (G)$.
Suppose that $b^{[4]}\bdot T \in \mathcal B (G)$.

\medskip
\noindent
CASE 1. $|\pi(T)| = 3$.

Then we have
\[
  T \sim a^{[2]}\bdot b^{[2]} \quad \mbox{ or } \quad T \sim a\bdot a^{2}\bdot b \,.
\]
Since $b^{[4]}\bdot T \in \mathcal B (G)$, it follows that we have the only case $T \sim a^{[2]}\bdot b^{[2]}$.
Then, since $b^{[2]}\bdot a\bdot b\bdot a\bdot b \in \mathcal B (G)$,
\[
  b^{[2]}\bdot T \sim b^{[2]}\bdot a^{[2]}\bdot b^{[2]} \,\, \mbox{ implies that } \,\, b^{[2]}\bdot T \in \big[b^{[2]}\bdot a^{[2]}\bdot b^{[2]}\big] \subset \mathcal B (G) \,.
\]

\medskip
\noindent
CASE 2. $|\pi(T)| = 2$.

Then $T$ is $\mathcal B (G)$-equivalent with one of the sequence described in (\ref{length 2}).
Since $b^{[4]}\bdot T \in \mathcal B (G)$, the possible choice of $T$ under $\mathcal B (G)$-equivalence is one the following sequences{\rm \,:}
\[
  \begin{aligned}
  a\bdot (ab)^{[\rm even]}, \quad a\bdot (a^{2}b)^{[\rm even]}, \quad a^{2}\bdot (ab)^{[\rm even]}, \quad a^{2}\bdot (a^{2}b)^{[\rm even]} \,,\\
  b\bdot (ab)^{[\rm odd_{\ge 3}]}, \quad b\bdot (a^{2}b)^{[\rm odd_{\ge 3}]}, \quad (ab)^{[\rm odd]}\bdot a^{2}b, \quad ab\bdot (a^{2}b)^{[\rm odd]} \,.
  \end{aligned}
\]
Then we obtain the following simple calculation, and it can cover all other cases{\rm \,:}
\begin{center}
$ab\bdot b\bdot ab\bdot b\bdot a, \quad a^{2}b\bdot b\bdot a^{2}b\bdot a\bdot b, \quad ab\bdot b\bdot ab\bdot a^{2}\bdot b, \quad a^{2}b\bdot b\bdot a^{2}b\bdot b\bdot a^{2}$,\\ $ab\bdot b\bdot ab\bdot b\bdot ab\bdot b, \quad a^{2}b\bdot b\bdot a^{2}b\bdot b\bdot a^{2}b\bdot b, \quad ab\bdot b\bdot a^{2}b\bdot b$,
\end{center}
which are all product-one sequences, and thus it follows that $b^{[2]}\bdot T \in \mathcal B (G)$.

\medskip
\noindent
CASE 3. $|\pi(T)| = 1$.

Then $\langle \supp(T) \rangle$ is abelian subgroup by Lemma \ref{3.6}.5. If $\langle \supp(T) \rangle = \langle 1_G \rangle$, then we are done because $T \in \mathcal B (G)$.

If $\langle \supp(T) \rangle = \langle b \rangle$, then $T = b^{[n]}$ for some $n \in \N$ by Lemma \ref{3.6}.4.
Since $b^{[4]}\bdot T \in \mathcal B (G)$, it follows that $n$ should be even number, and hence we are done.

By symmetry, we can obtain the same result whenever $\langle \supp(T) \rangle = \langle ab \rangle \mbox{ or } \langle a^{2}b \rangle$.

Suppose now that $\langle \supp(T) \rangle = \langle a \rangle$.
Then $T = a^{[n_1]}\bdot (a^{2})^{[n_2]}$ for some $n_1, n_2 \in \N_0$.

\begin{itemize}
\item[i)] $n_1 = 0$.

          In this case, $b^{[4]}\bdot (a^{2})^{[n_2]} \in \mathcal B (G)$ implies that $n_2 \ge 2$. It follows that $b^{[2]}\bdot (a^{2})^{[n_2]} \in \mathcal B (G)$.

\item[ii)] $n_2 = 0$.

          In this case, $b^{[4]}\bdot a^{[n_1]} \in \mathcal B (G)$ implies that $n_1 \ge 2$. It follows that $b^{[2]}\bdot a^{[n_1]} \in \mathcal B (G)$.

\item[iii)] $n_1 \ge 1 \und n_2 \ge 1$.

           To avoid the trivial case, we may assume that $n_1 \neq n_2$.

           If $|n_1 - n_2| = 1$, then, since any choice of $n_1, n_2$ can be reduced to the case
           \[
             n_1 = 1, n_2 = 2 \quad \mbox{ or } \quad n_1 = 2, n_2 = 1 \,,
           \]

           it suffices to verify only such two cases, and we have the followings{\rm \,:}
           \begin{center}
            $a\bdot b\bdot a^{2}\bdot a^{2}\bdot b^{[3]} \in \mathcal B (G) \quad \mbox{ implies that } \quad  a\bdot b\bdot a^{2}\bdot a^{2}\bdot b \in \mathcal B (G)$,\\
            $a\bdot a\bdot b\bdot a^{2}\bdot b^{[3]} \in \mathcal B (G) \quad \mbox{ implies that } \quad a\bdot a\bdot b\bdot a^{2}\bdot b \in \mathcal B (G)$.
           \end{center}

           If $|n_1 - n_2| \ge 2$, then it can be reduced to the case $n_1 = 0$ or $n_2 = 0$.
\end{itemize}

\noindent
Therefore, in any cases, we can obtain $b^{[2]}\bdot T \in \mathcal B (G)$, and it follows that $b^{[2]} \sim b^{[4]}$.
Moreover, $b \nsim b^{[3]}$ because $b^{[3]}\bdot (ab)^{[3]} \in \mathcal B (G), \mbox{ but } \,\, b\bdot (ab)^{[3]} \notin \mathcal B (G)$.
Hence
\[
  [b], \quad \big[b^{[2]}\big], \quad \big[b^{[3]}\big]
\]
are distinct $\mathcal B (G)$-equivalence classes in the class semigroup $\mathcal C \big(\mathcal B (G), \mathcal F (G)\big)$.
\end{proof}
\end{example}


\bigskip
\section{Arithmetic Properties of the monoid of product-one sequences} \label{5}
\bigskip

The goal of this section is to study the arithmetic of the monoid $\mathcal B (G)$ of product-one sequences. To do so we briefly recall some arithmetical concepts (details can be found in \cite{Ge-HK06a}).

Let $H$ be an atomic monoid and $a, b \in H$. If $a = u_1 \cdot \ldots \cdot u_k$, where $k \in \N$ and $u_1, \ldots, u_k \in \mathcal A (H)$, then $k$ is called the length of the factorization and $\mathsf L_H (a) = \mathsf L (a) = \{ k \in \N \mid a \ \text{has a factorization of length} \ k \} \subset \N$ is the {\it set of lengths} of $a$. As usual we set $\mathsf L (a) = \{0\}$ if $a \in H^{\times}$, and then
\[
\mathcal L (H) = \{ \mathsf L (a) \mid a \in H \}
\]
denotes the {\it system of sets of lengths} of $H$. Let $k \in \N$ and suppose that $H \ne H^{\times}$. Then
\[
\mathcal U_k (H) = \bigcup_{k \in L, L \in \mathcal L (H)} L \ \subset \ \N
\]
denotes the union of sets of lengths containing $k$, and we set
\[
\rho_k (H) = \sup \mathcal U_k (H) \,.
\]
If $L = \{m_1, \ldots, m_k \} \subset \Z$ is a finite subset of the integers, where $k \in \N$ and $m_1 < \ldots < m_k$, then $\Delta (L) = \{m_i - m_{i-1} \mid i \in [2,k]\} \subset \N$ is the set of distances of $L$. If all sets of lengths are finite, then
\[
\Delta (H) = \bigcup_{L \in \mathcal L (H)} \Delta (L)
\]
is the {\it  set of distances} of $H$. It is easy to check that $\min \Delta (H) = \gcd \Delta (H)$. If $\theta \colon H \to B$ is a transfer homomorphism between atomic monoids having finite sets of lengths, then   (\cite[Proposition 3.2.3]{Ge-HK06a})
\begin{equation}  \label{transfer}
\mathcal L (H) = \mathcal L (B) \ \text{ whence} \  \Delta (H)=\Delta (B) \ \text{and} \   \mathcal U_k (H) = \mathcal U_k (B) \ \ \text{for all $k \in \N$. }
\end{equation}
It is well-known that if $H$ is finitely generated, then then all unions $\mathcal U_k (H)$ of sets of lengths (in particular, all sets of lengths) and the set of distances $\Delta (H)$ are finite  (\cite[Theorem 3.1.4]{Ge-HK06a}).

\smallskip
We will study the system of sets of lengths and all further mentioned invariants for the monoid $\mathcal B (G)$. As usual, we set
\[
\mathcal L (G) = \mathcal L \big( \mathcal B (G) \big), \ \Delta (G) = \Delta \big( \mathcal B (G) \big), \ \mathcal U_k (G) = \mathcal U_k \big( \mathcal B (G) \big)
\]
and $\rho_k (G) = \rho_k \big( \mathcal B (G) \big)$ for all $k \in \N$. Since $\mathcal B (G)$ is finitely generated by Lemma \ref{3.2}, all these invariants are finite. The significance of the monoid $\mathcal B (G)$ for abelian groups stems from its relevance for general Krull monoids. Indeed, if $H$ is a commutative Krull monoid with class group $G$ such that every class contains a prime divisor, then there is a transfer homomorphism $\theta \colon H \to \mathcal B (G)$ and hence $\mathcal L (H)=\mathcal L (G)$. This transfer process from monoids of zero-sum sequences to more general monoids carries over to transfer Krull monoids (\cite{Ge16c}).

We study the structure of $\mathcal U_k (G)$ and of $\Delta (G)$ for general finite groups $G$ and we will derive canonical bounds for their size. Most results are known in the abelian case partly with different bounds (for a recent survey on the abelian case we refer to \cite{Sc16a}). However, if $G$ is non-abelian, then $\mathcal B (G)$ is not a transfer Krull monoid by Proposition \ref{3.4}. Thus the present results cannot be derived from the abelian setting. If $|G| \le 2$, then $\mathcal B (G)$ is factorial, $\mathcal L (G) = \{ \{k\} \mid k \in \N_0\}$, $\Delta (G) = \emptyset$, and $\mathcal U_k (G) = \{k\}$ for all $k \in \N$. To avoid annoying case distinctions we exclude this trivial case.

\bigskip
\centerline{\it Throughout this section, let $G$ be a  finite group with $|G| \ge 3$.}
\smallskip

\smallskip
Although the forthcoming proofs parallel the proofs given in the commutative setting there is a main difference. It stems from the fact that, in the non-abelian case, the embedding $\mathcal  B (G) \hookrightarrow \mathcal F (G)$ is not a divisor homomorphism. Thus there exist $U, V \in \mathcal B (G)$ such that $U$ divides $V$ in $\mathcal F (G)$, but not in $\mathcal B (G)$. Moreover, $U$ and $V$ can be atoms (e.g., if $G$ is the quaternion group, as discussed in Examples \ref{4.1}, then $U=I^{[4]} \in \mathcal A (G)$ and $V=I^{[4]} \bdot J^{[2]} \in \mathcal A (G)$ have this property).

\smallskip
\begin{lemma} \label{5.1}~
Let $k, \ell \in \N$ with $k < \ell$ and $U_1, \ldots, U_k, V_1, \ldots, V_{\ell} \in \mathcal A (G)$ such that $U_1 \bdot \ldots \bdot U_k = V_1 \bdot \ldots \bdot V_{\ell}$. There exist $\mu \in [1,k], \,\, \lambda, \lambda' \in [1,\ell] \mbox{ with } \lambda \neq \lambda', \mbox{ and } g_1,g_2 \in G$ such that $U_{\mu} = g_1 \bdot g_2 \bdot \ldots \bdot g_m$ with $m \ge 2$, $1_G = g_1 \ldots g_m$,  $g_1 \t V_{\lambda}$, and $g_2 \t V_{\lambda^{'}}$ in $\mathcal F (G)$.
\end{lemma}

\begin{proof}
Assume to the contrary that the assertion does not hold. Since $1_G$ is a prime element of $\mathcal B (G)$, we may suppose without restriction that all $U_i$ and $V_j$ have length at least two.  We set $U_{1} = g_1 \bdot g_2 \bdot \ldots \bdot g_m$ where $m \ge 2$ and $1_G = g_1 \ldots g_m$. Then $g_1 \bdot g_2$ divides $V_{\lambda}$ for some $\lambda \in [1, \ell]$, say $\lambda = 1$. Then we consider $g_2 \bdot g_3$. Since the assertion does not hold, it follows that $g_1 \bdot g_2 \bdot g_3$ divides $V_1$. Proceeding recursively we obtain that $U_1 \t V_1$, say $V_1 = U_1 \bdot V_1'$ with $V_1' \in \mathcal F (G)$. Thus we obtain the equation
\[
U_2 \bdot \ldots \bdot U_k = V_1' \bdot V_2  \bdot \ldots \bdot V_{\ell} \,.
\]
Proceeding in this way we end up with an equation of the form
\[
1_{\mathcal F (G)} = S_{1} \bdot \ldots \bdot S_{\eta}\bdot V_{\eta +1} \bdot \ldots \bdot V_{\ell} \,,
\]
where $\eta \le k$ and $S_1, \ldots, S_{\eta} \in \mathcal F (G)$, a contradiction to $k < \ell$.
\end{proof}

We consider the following property{\rm \,:}
\begin{itemize}
\item[{\bf P.}] If $U = g_1 \bdot \ldots \bdot g_{\ell} \in \mathcal A (G)$ and $g_1=h_1h_2$ with $h_1, h_2 \in G$, then $U' = h_1 \bdot h_2 \bdot g_2 \bdot \ldots \bdot g_{\ell}$ is either an atom or a product of two atoms at most.
\end{itemize}

\smallskip
\begin{remark} \label{5.2}~
Of course, every abelian group satisfies Property {\bf P} and the same is true for some non-abelian groups such as the quaternion group. However, for every $n \ge 9$, the dihedral group $D_{2n}$ does not have Property {\bf P} as the following example shows.
\end{remark}

\smallskip
\begin{example} \label{5.3}~
1. Let $G = Q_8$ be the quaternion group as discussed in Example \ref{4.1}. Then $\mathsf D (G) = 6$, and clearly any atom having length at most $4$ satisfies Property {\bf P}.
Any atom $U$ having the length $6$ has the form
\[
  U = g_{1}^{[4]}\bdot g_{2}^{[2]}\,, \quad \mbox{ where } \,\, g_1, g_2 \in \{ I, J, K ,-I, -J, -K \} \,\, \mbox{ with } \,\, g_2 \neq \pm g_1 ,
\]
and any atom $V$ having length $5$ is one of the following three types{\rm \,:}
\begin{itemize}
\item $V = g_1^{[3]}\bdot g_2\bdot g_3$, \quad where $g_1, g_2, g_3 \in \{ I, J, K ,-I, -J, -K \} \,\, \mbox{ with } \,\, g_2 \neq g_3 \und g_2, g_3 \neq \pm g_1$.

\smallskip
\item $V =  g_1^{[2]}\bdot g_2^{[2]}\bdot (-E)$, \quad where $g_1, g_2 \in \{ I, J, K ,-I, -J, -K \} \,\, \mbox{ with } \,\, g_2 \neq \pm g_1$.

\smallskip
\item $V = g_1\bdot (-g_1)\bdot g_2\bdot (-g_2)\bdot (-E)$, \quad where $g_1, g_2 \in \{ I, J, K \} \,\, \mbox{ with } \,\, g_1 \neq g_2$.
\end{itemize}
It is easy to check that $G$ satisfies Property {\bf P}.

\medskip
2. Let $G = D_{2n} = \langle a, b \t a^{n} = b^{2} = 1_G \und ba = a^{-1}b \rangle$ be the dihedral group, where $n \ge 9$. Consider the following sequences{\rm \,:}
\[
  U = b^{[2]}\bdot (ba)^{[3]}\bdot a^{2}\bdot ba^{5} \quad \und \quad U' = b^{[2]}\bdot (ba)^{[3]}\bdot ba^{2}\bdot ba^{4}\bdot ba^{5} \,.
\]
Then $U'$ is a product of $b^{[2]}$, $(ba)^{[2]}$, and $ba^{2}\bdot ba\bdot ba^{4}\bdot ba^{5}$, which are atoms. Now we need to verify that $U$ is an atom. Assume to the contrary that $U$ is not atom, say $U = W_1\bdot W_2$ with $W_1, W_2 \in \mathcal B (G)$, and let $W_1$ be an atom with $a^{2} \t W_1$.
Write $W_1 = g_1\bdot g_2\bdot \ldots \bdot g_{\ell}$ with $g_1 = a^{2} \und g_1g_2\ldots g_{\ell} = 1_G$. Then $\ell = 3 \, \mbox{ or } \ell = 5$.
If $\ell = 3$, then $g_2g_3 = a^{-2} \und g_2, g_3 \in \{ b, ba, ba^{5} \}$. Since $n \ge 9$, this is impossible.
If $\ell = 5$, then $W_2 = b^{[2]} \, \mbox{ or } \, W_2 = (ba)^{[2]}$ which implies that $W_1 = (ba)^{[3]}\bdot a^{2}\bdot ba^{5} \, \mbox{ or } \, W_1 = b^{[2]}\bdot ba\bdot a^{2}\bdot ba^{5}$. It is easy to check that $W_1$ is not an atom.
Thus $G$ does not have Property {\bf P}.
\end{example}

\smallskip
\begin{lemma} \label{5.4}~
Suppose that $G$ satisfies Property {\bf P}. Then for every $S \in \mathcal B (G)$ with $\max \Delta \big( \mathsf L (S)\big) \ge 2$, there exists $T \in \mathcal B (G)$ such that $|T| < |S|$ and $\max \Delta \big(\mathsf L (T)\big) \ge \max \Delta \big(\mathsf L (S)\big) - 1$.
\end{lemma}

\begin{proof}
Let $S \in \mathcal B (G)$ and $d = \max \Delta \big(\mathsf L (S)\big) \ge 2$. Then there are $k, \ell \in \N$ and $U_{1},\ldots, U_{k},V_{1},\ldots, V_{\ell} \in \mathcal A (G)$ such that
\[
  S = U_{1}\bdot \ldots \bdot U_{k} = V_{1}\bdot \ldots \bdot V_{\ell}
\]
with $\ell-k = d$ and $U_{1}\bdot \ldots \bdot U_{k}$ has no factorization of length in $[k+1,\ell-1]$.
Since $1_{G} \in \mathcal B (G)$ is a prime element, we may assume that $1_{G}$ does not divide $S$, and thus all $U_{i}$ and $V_{j}$ have length at least $2$.

By Lemma \ref{5.1}, we may assume that there are $g_{1},g_{2} \in G$ such that
\[
  g_{1}\bdot g_{2} \t U_{1}, \quad g_{1} \t V_{1} \,\, \und \,\, g_{2} \t V_{2} \,\, \mbox{ in } \mathcal F (G).
\]
Let $U_{1} = g_{1}\bdot g_{2}\bdot U'_{1}$, $V_{1} = g_{1}\bdot V'_{1} \und V_{2} = g_{2}\bdot V'_{2}$ with $U'_{1}, V'_{1}, V'_{2} \in \mathcal F (G)$.
Then we set $g_{0} = g_{1}g_{2} \in G$, and consider
\[
  U''_{1} = g_{0}\bdot U'_{1} \und V''_{1} = g_{0}\bdot V'_{1}\bdot V'_{2}.
\]
Clearly, $U''_{1} \in \mathcal A (G)$. Since $V_1$ gives rise to a product-one equation with $g_1$ being the finial element and $V_2$ gives rise to a product-one equation with $g_2$ being the first element, it follows that $V''_{1} \in \mathcal B (G)$. We obtain that
\[
  U''_{1}\bdot U_{2} \bdot \ldots \bdot U_{k} = V''_{1}\bdot V_{3}\bdot \ldots \bdot V_{\ell}.
\]
We set $T = U''_{1}\bdot U_{2}\bdot \ldots \bdot U_{k}$ and observe that $|T| < |S|$. If $T = W_{1}\bdot \ldots \bdot W_{t}$ with $W_{1},\ldots,W_{t} \in \mathcal A (G)$ and $g_{0} \t W_{1}$ in $\mathcal F (G)$, then $W_{1} = g_{0}\bdot W'_{1}$ and $W''_{1} = g_{1}\bdot g_{2}\bdot W'_{1}$ is either atom or product of precisely two atoms. Thus $S = U_{1}\bdot \ldots \bdot U_{k}$ has a factorization of length $t$ or $t+1$. It follows that $T$ has no factorization of length in $[k+1, \ell-2]$, and thus we obtain
\[
  \max \Delta \big(\mathsf L (T)\big) \ge \ell - 1 - k = d - 1 = \max \Delta \big(\mathsf L (S)\big)-1 \,.
\]
\end{proof}

The next result shows that the set of distances and all unions of sets of lengths of $\mathcal B (G)$ are finite intervals. This is far from being true in general. Indeed, for every finite set $\Delta \subset \N$ with $\min \Delta = \gcd \Delta$ there is a finitely generated Krull monoid $H$ such that $\Delta (H ) = \Delta$ (\cite{Ge-Sc17a}).

\smallskip
\begin{theorem} \label{5.5}~

\begin{enumerate}
\item $\mathcal U_k (G)$ is a finite interval for all $k \in \N$.

\smallskip
\item If $G$ satisfies Property {\bf P}, then $\Delta (G)$ is a finite interval with $\min \Delta (G)=1$.
\end{enumerate}
\end{theorem}

\begin{proof}
1. Let $k \in \N$. We set  $\lambda_{k} (G) = \min \mathcal U_{k} (G)$. Then $\mathcal U_{k}(G) \subset [\lambda_k (G), \rho_k (G)]$, and we have to show equality. Note that it is sufficient to prove that $[k,\rho_{k} (G)] \subset \mathcal U_{k} (G)$. Indeed, suppose that this is done, and let $\ell \in [\lambda_{k} (G), k]$. Then $\ell \le k \le \rho_{\ell} (G)$, hence $k \in \mathcal U_{\ell} (G)$ and consequently $\ell \in \mathcal U_{k} (G)$. It follows that $[\lambda_k (G), \rho_k (G)] = \mathcal U_k (G)$.

To prove the assertion, let $\ell \in [k,\rho_{k} (G)]$ be  minimal such that $[\ell,\rho_{k} (G)] \subset \mathcal U_{k} (G)$.
Assume to the contrary that $\ell > k$.
We set $\Omega = \{A\in \mathcal B (G) \t \{k,j\} \subset \mathsf L (A) \, \textnormal{ for some } \, j \geq \ell\}$ and choose $B \in \Omega$ such that $|B|$ is minimal.
Then
\[
  B = U_{1}\bdot \ldots \bdot U_{k} = V_{1}\bdot \ldots \bdot V_{j}, \quad \text{where $U_1, \ldots, U_k, V_1, \ldots, V_j \in \mathcal A(G)$} \,,
\]
and by Lemma \ref{5.1}, we may assume that
\[
  U_{k} = g_{1} \bdot g_{2} \bdot U \,\, \mbox{ with } \,\, g_{1} \t V_{j-1} \und g_{2} \t V_{j} \,,
\]
where $U = g_{3} \bdot \ldots \bdot g_{t}$ and $g_{1}\ldots g_{t} = 1_{G}$.
We set $g_{0} = g_{1}g_{2} \in G$ and obtain that $U'_{k} = g_{0} \bdot U \in \mathcal A (G)$.
Write
\[
  V_{j-1} = h_{1} \bdot \ldots \bdot h_{n} \bdot g_{1} \und V_{j} = g_{2} \bdot s_{1} \bdot \ldots \bdot s_{m}
\]
with $h_{1}\ldots h_{n}g_{1} = 1_{G}$ and $g_{2}s_{1}\ldots s_{m} = 1_{G}$.
Let $V'_{j-1} = g_{0}\bdot V \in \mathcal B (G)$, where $V = h_{1}\bdot \ldots \bdot h_{n}\bdot s_{1}\bdot \ldots \bdot s_{m} \in \mathcal F (G)$, and consider
\[
  B' = U_{1}\bdot \ldots \bdot U_{k-1}\bdot U'_{k} \,.
\]
Then $|B'| < |B|$, and
\[
  B' = U_{1}\bdot \ldots \bdot U_{k-1}\bdot U'_{k} = V_{1}\bdot \ldots \bdot V_{j-2}\bdot W_{1}\bdot \ldots \bdot W_{t} \,,
\]
where $V'_{j-1} = W_{1}\bdot \ldots \bdot W_{t}$ with $W_{1},\ldots , W_{t} \in \mathcal A (G)$.
By the minimality of $|B|$, we have $(j-2)+t < \ell \le j$. Hence $t=1$ and $j = \ell$.
Thus $\ell -1 \in \mathcal U_{k} (G)$, a contradiction to the minimality of $\ell$.
Therefore $\ell = k$ and hence $[k,\rho_{k}(G)] \subset \mathcal U_{k} (G)$.

\smallskip
2. Since $\Delta (G)$ is finite as mentioned before, we can take $S_{0} \in \mathcal B (G)$ with minimal length such that $\max \Delta \big(\mathsf L (S_{0})\big) = \max \Delta (G)$.
By Lemma \ref{5.4}, we can find $S_{1} \in \mathcal B (G)$ with minimal length such that $|S_{1}| < |S_{0}|$ and $\max \Delta \big(\mathsf L (S_{1})\big) \ge \max \Delta (G)-1$.
By the minimality of $S_{0}$,
\[
   |S_{1}| < |S_{0}| \quad \quad \mbox{ implies } \quad \quad \max \Delta \big(\mathsf L (S_{1})\big) < \max \Delta(G) \,.
\]
It follows that $\max \Delta (G)-1 = \max \Delta \big(\mathsf L (S_{1})\big) \in \Delta(G)$.
Again Lemma \ref{5.4} implies that we can find $S_{2} \in \mathcal B (G)$ with minimal length such that $|S_{2}| < |S_{1}|$ and $\max \Delta \big(\mathsf L (S_{2})\big) \ge \max \Delta (G) -2$. By the minimality of $S_{1}$,
\[
 |S_{2}| < |S_{1}| \quad \quad \mbox{ implies } \quad \quad \max \Delta \big(\mathsf L (S_{2})\big) < \max \Delta(G) -1 \,.
\]
It follows that $\max \Delta(G) -2 = \max \Delta \big(\mathsf L (S_{2})\big) \in \Delta(G)$. Continuing this process, we can obtain $S_{n} \in \mathcal B (G)$ such that
\[
   1 = \max \Delta (G) - n = \max \Delta \big(\mathsf L (S_{n})\big) \in \Delta(G) \,.
\]
Thus $\Delta(G) = [1,\max \Delta (G)]$ is an interval.
\end{proof}

\smallskip
Next we study the maxima of the sets $\mathcal U_k (G)$. Recall that we have defined $\rho_k (G) = \max \mathcal U_k (G)$ for all $k \in \N$.  Even in case of abelian groups, precise values of $\rho_{2k+1} (G)$ are unknown in general (\cite{Fa-Zh16a}).

\smallskip
\begin{proposition} \label{5.6}~
Let $k \in \N$.
\begin{enumerate}
\item $\rho_k (G) \le \frac {k \mathsf D (G)}{2}$ for all $k \in \N$.

\smallskip
\item $\rho_{2k} (G) = k \mathsf D (G)$ for all $k \in \N$.
\end{enumerate}
In particular, $1 + k \mathsf D (G) \le \rho_{2k+1} (G) \le k \mathsf D (G) + \frac {\mathsf D (G)}{2}$ for all $k \in \N$.
\end{proposition}

\begin{proof}
1. Let $k \in \N$.  Let $A \in \mathcal B (G)$ with $k \in \mathsf L (A)$. We have to show that $\max \mathsf L (A) \le \frac {k \mathsf D (G)}{2}$. There are $U_{1}, \ldots, U_{k} \in \mathcal A (G)$ such that
\[
  A = U_{1}\bdot \ldots \bdot U_{k} \,.
\]
Since $1_{G} \in \mathcal B (G)$ is a prime element, we may assume that $1_{G}$ does not divide $A$, and hence any atom dividing $A$ has length at least $2$.
If $A = V_{1} \bdot \ldots \bdot V_{\ell}$ for $V_{1}, \ldots, V_{\ell} \in \mathcal A (G)$, then we obtain
\[
2\ell \le \sum_{i=1}^{\ell} |V_{i}| = |A| = \sum_{j=1}^{k} |U_{j}| \le k \mathsf D(G) \,.
\]
It follows that $\ell \le \frac {k \mathsf D (G)}{2}$, and thus $\max \mathsf L (A) \le \frac {k \mathsf D (G)}{2}$.

\smallskip
2. Let $k \in \N$. Then, by 1., we have $\rho_{2k} (G) \le k \mathsf D (G)$.
Now let $U = g_{1} \bdot \ldots \bdot g_{\ell} \in \mathcal A (G)$ with $\ell = |U| = \mathsf D (G)$, where $g_{1},\ldots, g_{\ell} \in G$.
Consider the  sequence
\[
V = g^{-1}_{1} \bdot \ldots \bdot g^{-1}_{\ell} \in \mathcal A (G) \,.
\]
Then $k \mathsf D (G) \in \mathsf L \big((U\bdot V)^{[k]}\big)$. Since $2k \in \mathsf L \big((U\bdot V)^{[k]}\big)$, we obtain $\rho_{2k} (G) \ge k \mathsf D (G)$, and hence $\rho_{2k} (G) = k \mathsf D(G)$.

It remains to prove the ``In particular" statement.  For all $i, j \in \N$, we have $\mathcal U_i (G) + \mathcal U_j (G) \subset \mathcal U_{i+j}(G)$ whence $\rho_{i} (G) + \rho_{j} (G) \le \rho_{i+j} (G)$.  Thus the ``In particular" statement   follows from 1. and 2.
\end{proof}

\medskip
Our next goal is to study the maximum of the set of distances $\Delta (G)$. If $G$ is abelian, then precise values for $\max \Delta (G)$ are known only for very special classes of groups including cyclic groups (\cite{Ge-Zh15b}). For general groups $G$ we study $\max \Delta (G)$ via a further invariant (introduced in Definition \ref{5.7}). The main result is Proposition \ref{5.12}.

\smallskip
\begin{definition} \label{5.7}~
Let $H$ be an atomic monoid.
\begin{enumerate}
\item For $b \in H$, let $\omega(H,b)$ denote the smallest $N \in \N_{0} \cup \{\infty\}$ with the following property{\rm \,:}
            \begin{center}
            For all $n \in \N$ and $a_{1},\ldots, a_{n} \in H$, if $b \mid a_{1}\ldots a_{n}$, then there exists a subset $\Omega \subset [1,n]$ such that $| \Omega | \leq N$ and $b \mid \prod_{\nu \in \Omega}a_{\nu}$.
            \end{center}

\smallskip
\item We set
      \[
      \omega(H) = \sup\,\{\omega(H,u) \mid u \in \mathcal A (H)\} \in \N_{0} \cup \{\infty\}.
      \]
\end{enumerate}
\end{definition}

\smallskip
If $H$ is an atomic monoid with $\Delta (H) \ne \emptyset$, then $2+\max \Delta (H) \le \omega (H)$ (see \cite[Theorem 1.6.3]{Ge-HK06a} and  \cite[Proposition 3.6]{Ge-Ka10a}).
As usual we set $\omega (G) = \omega \big( \mathcal B (G) \big)$ and observe that $2+\max \Delta (G) \le \omega (G)$. If $G$ is a cyclic group or an elementary $2$-group, then $2+\max \Delta (G) = \omega (G) = \mathsf D (G)$.

\smallskip
\begin{lemma} \label{5.8}~
We have $\mathsf D (G) \le \omega (G)$, and equality holds if $G$ is abelian.
\end{lemma}

\begin{proof}
Let $U = g_{1}\bdot \ldots \bdot g_{\ell} \in \mathcal A (G)$ with $\ell = |U| = \mathsf D (G)$.
Put $V_{i} = g_{i}\bdot g^{-1}_{i} \in \mathcal A (G)$ for all $i \in [1,\ell]$. Then we obtain
\[
  U \t V_{1}\bdot \ldots \bdot V_{\ell} \,\, \mbox{ in } \mathcal B (G) \,.
\]
Assume to the contrary that $U$ divides a proper subproduct in $\mathcal B (G)$, say $U \t V_{1}\bdot \ldots \bdot V_{k}$ for some $k \in [1,\ell-1]$.
Then $2k \ge \ell$. If $2k = \ell$, then $U = V_{1}\bdot \ldots \bdot V_{k}$ is not atom in $\mathcal B (G)$.
If $2k > \ell$, then $\ell-k < k$, and, for each $j = [1,\ell-k]$, we may assume that $g_{k+j} = g_{j}^{-1}$.
Hence we have
\[
  U = (g_{1}\bdot g_{k+1})\bdot \ldots \bdot (g_{\ell-k}\bdot g_{\ell})\bdot (g_{\ell-k+1}\bdot \ldots \bdot g_{k}) \,,
\]
and each terms is in $\mathcal B (G)$. It follows that $U$ is not atom in $\mathcal B (G)$, a contradiction.
Thus $U$ cannot divide a proper subproduct, and hence $\mathsf D (G) \le \omega (G)$.

Suppose that $G$ is abelian. It remains to show that $\omega (G) \le \mathsf D (G)$. Let $U \in \mathcal A (G)$.
If $U \t V_{1}\bdot \ldots \bdot V_{k}$ for $V_{1},\ldots, V_{k} \in \mathcal B (G)$, then there is a subproduct of at most $|U|$ factors such that $U$ divides this subproduct in $\mathcal F (G)$. Since $\mathcal B (G) \hookrightarrow \mathcal F (G)$ is a divisor homomorphism, this divisibility holds in $\mathcal B (G)$.
Hence $\omega (G, U) \le |U| \le \mathsf D (G)$, and it follows that $\omega (G) \leq \mathsf D (G)$.
\end{proof}

\smallskip
Next we define Davenport constants of additive commutative semigroups (as studied in \cite{Wa-Ga08a, Ad-Ga-Wa14a, Wa15a}) and will apply these concepts to the class semigroup of $\mathcal B (G)$. Recall that all our semigroups have an identity element (a zero-element in the additive case) and we use the convention that an empty sum equals the zero-element.

\medskip
\begin{definition} \label{5.9}~
Let $\mathcal C$ be an additive commutative semigroup.

\begin{enumerate}
\item We denote by $\mathsf d (\mathcal C)$ the smallest $d \in \N_{0} \cup \{\infty\}$ with the following property :
            \begin{center}
             For all $n \in \N$ and $c_{1}, \ldots, c_{n} \in \mathcal C$, there exists $\Omega \subset [1,n]$ such that $| \Omega | \leq d$ and $\sum_{\nu=1}^{n}c_{\nu} = \sum_{\nu \in \Omega}c_{\nu}$.
            \end{center}

\smallskip
  \item We denote by $\mathsf D(\mathcal C)$ the smallest $\ell \in \N \cup \{\infty\}$ with the following property :
            \begin{center}
            For all $n \in \N$ with  $n \geq \ell$ and $c_{1}, \ldots, c_{n} \in \mathcal C$, there exists $\Omega \subsetneq [1, n]$ such that $\sum_{\nu=1}^{n} c_{\nu} = \sum_{\nu \in \Omega}c_{\nu}$.
            \end{center}
  \end{enumerate}
\end{definition}

If $G$ is a finite abelian group, then of course the definitions for $\mathsf D (G)$ and $\mathsf d (G)$ given in Definition \ref{3.1} and in Definition \ref{5.9} coincide.

\smallskip
\begin{lemma} \label{5.10}~
If $\mathcal C$ is a finite commutative semigroup, then $\mathsf D (\mathcal C) = \mathsf d ( \mathcal C)+1 \le |\mathcal C|$.
\end{lemma}

\begin{proof}
The equality $\mathsf D (\mathcal C) = \mathsf d ( \mathcal C)+1$ is verified in \cite[Proposition 1.2]{Ad-Ga-Wa14a}. Since we could not find a reference that $|\mathcal C|$ is an upper bound, we provide the short argument. To show that $\mathsf D (\mathcal C) \le |\mathcal C|$, it suffices to verify that $|\mathcal C|$ has the property given in Definition \ref{5.9}. So let $n \in \N$ with  $n \geq |\mathcal C|$ and $c_{1}, \ldots, c_{n} \in \mathcal C$. If all sums
\[
0 = \sum_{\nu \in \emptyset} c_{\nu}, c_1, c_1+c_2, \ldots, c_1+ \ldots + c_{n-1}
\]
are pairwise distinct, then one of the sums coincides with $c_1+ \ldots + c_n$ as required. Suppose there are $k, \ell \in [0, n-1]$ with $k < \ell$ and $c_1+ \ldots + c_k = c_1 + \ldots + c_{\ell}$. Then
$c_1+ \ldots + c_n = c_1 + \ldots + c_k + c_{\ell + 1 } + \ldots + c_n$ as required.
\end{proof}

\smallskip
\begin{lemma} \label{5.11}~
Let $F$ be a monoid, $H \subset F$ a submonoid with $H^{\times} = H \cap F^{\times}$ and $\mathcal C = \mathcal C^* (H,F)$.

If $n \in \N \und f, a_{1}, \ldots, a_{n} \in F \mbox{ with } fa_{1}\ldots a_{n} \in H$, then there exists a subset $\Omega \subset [1,n]$ such that
\[
  |\Omega| \le \mathsf d (\mathcal C) \und f\prod_{\nu \in \Omega}a_{\nu} \in H \,.
\]
\end{lemma}

\begin{proof}
By the definition of $\mathsf d (\mathcal C)$, there exists a subset $\Omega \subset [1,n]$ such that $|\Omega| \le \mathsf d (\mathcal C)$ and
\[
\sum_{\nu = 1}^{n}[a_{\nu}]^{F}_{H} = \sum_{\nu \in \Omega}[a_{\nu}]^{F}_{H}.
\]
Since $fa_{1}\ldots a_{n} \in H$, we have
\[
f\prod_{\nu \in \Omega}a_{\nu} \in \Big[f\prod_{\nu \in \Omega}a_{\nu}\Big]^{F}_{H} = [f]^{F}_{H} + \sum_{\nu \in \Omega}[a_{\nu}]^{F}_{H} = [f]^{F}_{H} + \sum_{\nu = 1}^{n}[a_{\nu}]^{F}_{H} = \Big[f\prod_{\nu = 1}^{n}a_{\nu}\Big]^{F}_{H} \subset H \,. \qedhere
\]
\end{proof}

\smallskip
\begin{proposition} \label{5.12}~

\begin{enumerate}
\item $\mathsf D (G) \le \omega (G) \le \mathsf D (G) + \mathsf d \Big( \mathcal C \big(\mathcal B (G), \mathcal F (G)\big) \Big)$.

\smallskip
\item $\mathsf D (G/G') \le \mathsf D \Big( \mathcal C \big(\mathcal B (G), \mathcal F (G)\big) \Big)$, and equality holds if $G$ is abelian.
\end{enumerate}
\end{proposition}

\begin{proof}
We set $\mathcal C = \mathcal C (\mathcal B (G), \mathcal F (G) \big)$.

1. The first inequality follows from Lemma \ref{5.8}. For the second inequality, let $U \in \mathcal A (G)$, and $A_{1}, \ldots, A_{n} \in \mathcal B (G)$ such that
\[
  U \t A_{1}\bdot \ldots \bdot A_{n} \,\, \mbox{ in } \mathcal B (G) \,.
\]
After renumbering if necessary, we may assume that $U \t A_{1}\bdot \ldots \bdot A_{|U|}$ in $\mathcal F (G)$, say $A_{1}\bdot \ldots \bdot A_{|U|} = U\bdot W$ with $W \in \mathcal F (G)$.
Then $W\bdot A_{|U|+1}\bdot \ldots \bdot A_{n} \in \mathcal B (G)$. By Lemma \ref{5.11}, there is a subset $\Omega \subset \big[|U|+1, n\big]$, say $\Omega = \big[|U|+1,m\big]$, such that $|\Omega| \leq \mathsf d (\mathcal C) \und W\bdot A_{|U|+1}\bdot \ldots \bdot A_{m} \in \mathcal B (G)$. Thus we obtain
\[
  A_{1}\bdot \ldots \bdot A_{m} = U\bdot (W\bdot A_{|U|+1}\bdot \ldots \bdot A_{m})
\]
and $m \le |U| + \mathsf d (\mathcal C) \le \mathsf D (G) + \mathsf d (\mathcal C)$.

\smallskip
2. By Theorem \ref{3.8}.1, $G/G'$ is an isomorphic to a subsemigroup of $\mathcal C$. Clearly, this implies that $\mathsf D (G/G') \le \mathsf D (\mathcal C)$. Suppose that $G$ is abelian. Then $|G'|=1$ and $G/G'$ is isomorphic to $G$. Furthermore, $\mathcal B (G) \hookrightarrow \mathcal F (G)$ is a divisor theory, the class semigroup $\mathcal C$ is isomorphic to the class group of $\mathcal B (G)$ by Lemma \ref{2.1}, and the class group of $\mathcal B (G)$  is isomorphic to $G$. Thus equality holds.
\end{proof}

\providecommand{\bysame}{\leavevmode\hbox to3em{\hrulefill}\thinspace}
\providecommand{\MR}{\relax\ifhmode\unskip\space\fi MR }
\providecommand{\MRhref}[2]{%
  \href{http://www.ams.org/mathscinet-getitem?mr=#1}{#2}
}
\providecommand{\href}[2]{#2}

\end{document}